\documentclass[critical zeros]{article}
\usepackage{amsmath, amssymb, amsthm, delarray}

\usepackage[top =20mm, left=35mm,right=37mm]{geometry}

\newtheorem{thm}{Theorem}[section]
\newtheorem*{theorem*}{Theorem}
\newtheorem{lem}[thm]{Lemma}
\newtheorem{prop}[thm]{Proposition}

\theoremstyle{definition}

\numberwithin{equation}{section}
\theoremstyle{remark}

\begin{document}

\title{\textbf{Zeros of  Dirichlet $L$-functions on the critical line}}

\author{Keiju Sono}

\date{}
\allowdisplaybreaks

\maketitle 
\noindent
\begin{abstract}
In this paper, we estimate the proportion of zeros of Dirichlet $L$-functions on the critical line. Using Feng's mollifier \cite{F} and an asymptotic formula for the mean square of Dirichlet $L$-functions introduced in \cite{CIS3}, we prove that averaged over primitive characters and conductors,  at least 61.07 \% of zeros of Dirichlet $L$-functions are on the critical line, and  at least 60.44 \% of zeros are simple and on the critical line. These results are an improvement on the work of Conrey, Iwaniec and Soundararajan in \cite{CIS2}.

\end{abstract}

\section{Introduction}
Let $q$ be a positive integer and  $\chi$ be an even primitive character modulo $q$. The Dirichlet $L$-function associated to $\chi$ is defined by
\begin{equation}
\label{1}
L(s, \chi)=\sum _{n=1}^{\infty}\frac{\chi(n)}{n^{s}}
\end{equation}
for $\Re (s)>1$ and continued holomorphically to the whole complex plane. Put
\begin{equation}
\label{2}
\Lambda \left( \frac{1}{2}+s, \chi\right)=\left( \frac{q}{\pi} \right)^{s/2}\Gamma \left( \frac{1}{4}+\frac{s}{2} \right)L\left( \frac{1}{2}+s, \chi \right).
\end{equation}
Then $L(s, \chi)$ satisfies the functional equation 
\begin{equation}
\label{3}
\Lambda (s, \chi)=\varepsilon_{\chi}\Lambda (1-s, \overline{\chi}),
\end{equation}
where $|\varepsilon_{\chi}|=1$. For $T>0$, we denote the number of zeros $\rho =\beta +i\gamma$ of the Dirichlet $L$-function $L(s, \chi)$ with $0<\beta <1$, $T\leq \gamma <2T$ by $N(T, \chi)$. It is well known that 
\begin{equation}
\label{4}
N(T, \chi )=\frac{T}{2\pi}\log \frac{2qT}{\pi e} +O(\log qT).
\end{equation}
We denote the number of these zeros with $\beta =1/2$ by $N_{0}(T, \chi)$, and denote by $N_{0}^{'}(T,\chi)$ the number of simple zeros counted by $N_{0}(T, \chi)$. The Generalized Riemann Hypothesis for $L(s, \chi)$ implies that $N_{0}(T, \chi)=N(T, \chi)$ for any $T>0$.    Furthermore, it is widely expected that $N_{0}^{'}(T, \chi)=N(T, \chi)$, i.e., all non-trivial zeros of $L(s, \chi)$ are expected to be simple and on the critical line.  There are several results of types
\begin{equation}
\label{propocrit}
N_{0}(T, \chi)\geq \kappa \; N(T, \chi)
\end{equation}
or
\begin{equation}
\label{proposim}
N_{0}^{'}(T, \chi) \geq \kappa ^{'}N(T, \chi).
\end{equation}
In the case of the Riemann zeta-function, Selberg \cite{S} first proved that (\ref{propocrit}) holds for some small constant $\kappa >0$. This means that positive proportion of non-trivial zeros of $\zeta (s)$ are on the critical line. Levinson \cite{L} showed that (\ref{proposim}) holds with $\kappa ^{'}=0.3474$.  Conrey \cite{C} improved on Levinson's result and proved that (\ref{propocrit}) holds with $\kappa =0.4077$ and (\ref{proposim}) holds with $\kappa ^{'}=0.401$. In the case of Dirichlet $L$-functions, Wu \cite{W2} generalized Conrey's result and proved that (\ref{propocrit}) holds with $\kappa =0.4172$ and (\ref{proposim}) holds with $\kappa ^{'}=0.4074$ uniformly in $q$ with $\log q =o(\log T)$ for sufficiently large $T$.

In the case of the estimation of the proportion of critical zeros of Dirichlet $L$-functions,  one can obtain some better results by averaging over characters and conductors. Let $\Psi (x)$ be a  non-negative 
 smooth function compactly  supported in $[1,2]$. Put 
 \[
 \mathfrak{N}(T, Q)=\sum _{q}\frac{\Psi (q/Q)}{\varphi (q)}\sum _{\chi (\mathrm{mod}\; q)}^{\quad \quad *}N(T, \chi)
 \]
 for $Q, T \geq 3$. Here, $\varphi$ denotes Euler's totient function and the subscript $*$ means that the summation is restricted to all primitive characters. Let $\mathfrak{N}_{0}(T, Q)$ (resp. $\mathfrak{N}_{0}^{'}(T, Q)$) denote the same sum, but with $N(T, \chi)$  replaced by $N_{0}(T, \chi)$ (resp. $N_{0}^{'}(T, \chi)$). Conrey, Iwaniec and Soundararajan \cite{CIS2} proved that for $Q$ and $T$ with $(\log Q)^{6} \leq T \leq (\log Q)^{A}$,
 \begin{equation}
 \label{CIS}
 \mathfrak{N}_{0}^{'}(T, Q) \geq \frac{14}{25}\; \mathfrak{N}(T,Q)
 \end{equation}
 holds, where $A\geq 6$ is any fixed constant, and $Q$ is sufficiently large in terms of $A$. This implies that at least 56\% of zeros of a family of Dirichlet $L$-functions  in the domain $(\log Q)^{6}\leq \Im (s) \leq (\log Q)^{A}$ are simple and on the critical line.  In that paper they used Levinson's mollifier associated to a polynomial $P(x)=x$. They mentioned that by choosing the best possible polynomial,   one can prove that the proportion of simple critical zeros  is at least 58.65\%. In the same situation, Wu \cite{W} proved that more than 80.12\% of zeros of a family of Dirichlet $L$-functions are distinct, and that more than 60.24\% of zeros of a family of Dirichlet $L$-functions are simple. In addition, he improved on these proportions to 83.21\% and 66.43\%  assuming the Generalized Riemann Hypothesis   .
 
 In this paper, to estimate the proportion of critical zeros, we use slightly different definitions. Let $W(x)$ be a smooth function compactly supported in $[1,2]$.  For $Q, T \geq 3$, put 
 \begin{equation}
 \label{5}
 {\cal N}(T, Q)=\sum _{q}W\left( \frac{q}{Q} \right) \sum _{\chi (\mathrm{mod}\; q)}^{\quad \quad \flat}N(T, \chi),
 \end{equation}
where the subscript $\flat$ means that the summation is restricted to all even primitive characters.  We also  define ${\cal N}_{0}(T, Q)$ and ${\cal N}_{0}^{'}(T,Q)$ by replacing $N(T, \chi)$ in (\ref{5}) with $N_{0}(T, \chi)$, $N_{0}^{'}(T, \chi)$, respectively. In \cite{CIS3}, Conrey, Iwaniec and Soundararajan established a formula of mean square of the product of a Dirichlet $L$-function and a Dirichlet polynomial of length $Q^{\theta}$ with $0<\theta <1$ (see Theorem 1.2 below). This formula enables us to compute mollified moments of Dirichlet $L$-functions with mollifier of length $Q^{\theta}$, $0<\theta <1$. In the same paper they mentioned without proof that by using Feng's mollifier instead of Levinson's,  one can improve the results of \cite{CIS2} and obtain 
\begin{equation}
\label{CIScritical}
{\cal N}_{0}(T, Q) \geq \frac{3}{5}\;  {\cal N}(T,Q).
\end{equation}
The value $3/5$ is close to the lower bound for the proportion of critical zeros of the Riemann zeta-function assuming so called the $\theta =1$ conjecture (see \cite{F}). The main purpose of this paper is to establish a complete proof of the estimation of type (\ref{CIScritical}).  We prove the following theorem.
\begin{thm}
For $(\log Q)^{2} \leq T \leq (\log Q)^{A}$, we have
\begin{equation}
\label{mtsimple}
{\cal N}_{0}^{'}(T, Q) \geq 0.6044\; {\cal N}(T, Q),
\end{equation}
and
\begin{equation}
\label{mtcritical}
{\cal N}_{0}(T, Q) \geq 0.6107\;  {\cal N}(T, Q),
\end{equation}
where $A >2$ is an arbitrary fixed constant and $Q$ is sufficiently large in terms of $A$.
\end{thm}
Instead of the asymptotic large sieve, we use the following result on mean square of Dirichlet $L$-functions with shifts.
\begin{thm}[Conrey, Iwaniec, Soundararajan \cite{CIS3}]
Let $Q>2$ be sufficiently large. For positive integers $h, k$ and $\alpha ,\beta \in \mathbb{C}$, put
\begin{equation}
\label{29}
\Delta _{\alpha ,\beta }(h,k;Q):=\sum _{q}W\left( \frac{q}{Q} \right) \sum _{\chi (\mathrm{mod}\; q)}^{\quad \quad \flat} \Lambda \left(\frac{1}{2}+\alpha ,\chi \right)\Lambda \left(\frac{1}{2}+\beta , \overline{\chi} \right)\chi (h) \overline{\chi}(k).
\end{equation}
Suppose that the shifts $\alpha$ and $\beta$ are $ \ll 1/\log Q$. Then, 
\begin{equation}
\label{31}
\begin{aligned}
& \Delta _{\alpha ,\beta }(h,k;Q) \\
& \quad =\sum _{\underset{(q,hk)=1}{q}} W\left( \frac{q}{Q} \right)   \varphi ^{\flat}(q) \\
&\quad \quad \times\left\{ \left( \frac{q}{\pi}\right)^{\frac{\alpha +\beta }{2}}\Gamma \left(\frac{1}{4}+\frac{\alpha}{2}\right)\Gamma \left(\frac{1}{4}+\frac{\beta}{2}\right) \frac{(h,k)^{1+\alpha +\beta}}{h^{1/2+\beta}{k^{1/2+\alpha}}}\zeta _{q}(1+\alpha +\beta ) \right. \\
& \left. \quad \quad \quad \quad  +\left( \frac{q}{\pi} \right)^{\frac{-\alpha -\beta }{2}}\Gamma \left(\frac{1}{4}-\frac{\alpha}{2}\right)\Gamma \left( \frac{1}{4}-\frac{\beta}{2}\right) \frac{(h,k)^{1-\alpha -\beta}}{h^{1/2-\alpha}{k^{1/2-\beta}}}\zeta _{q}(1-\alpha -\beta ) \right\} +{\cal E}_{h,k},
\end{aligned}
\end{equation}
where $\zeta _{q}(s):=\zeta (s)\prod _{p|q}(1-1/p^{s})$ and the remainder terms ${\cal E}_{h,k}$ satisfy
\begin{equation}
\label{32}
\sum _{h, k \leq Q^{\theta}}\frac{\lambda _{h}\overline{\lambda _{k}}}{\sqrt{hk}}{\cal E}_{h,k}=o(Q^{2})
\end{equation}
for any $0<\theta <1$ and any sequence $(\lambda _{h})$ with $\lambda _{h} \ll h^{\varepsilon}$. 
\end{thm}
In the statement of the above theorem, the ranges of $\alpha$ and $\beta$ are restricted to $\alpha , \beta \ll 1/\log Q$.   To evaluate the proportion of critical zeros away from the real axis, we expect that the asymptotic formula (\ref{31}) is still valid for larger values of $|\Im (\alpha)|$, $|\Im (\beta)|$. Since the main terms are multiplied by  gamma factors,  the remainder terms ${\cal E}_{h,k}$ in this case are expected to satisfy
\begin{equation}
\label{32.5}
\sum _{h, k \leq Q^{\theta}}\frac{\lambda _{h}\overline{\lambda _{k}}}{\sqrt{hk}}{\cal E}_{h,k}=o\left(Q^{2} \left|\Gamma \left( \frac{1}{4}+i\frac{\Im (\alpha )}{2} \right)  \Gamma \left( \frac{1}{4}+i\frac{\Im (\beta )}{2} \right) \right|\right).
\end{equation}
Following the argument in \cite{CIS3} with a minor adjustment, we easily see that the formula (\ref{31}) with (\ref{32.5}) is valid for $|\Re (\alpha)|, |\Re (\beta)| \ll 1/\log Q$, $|\Im (\alpha)|, |\Im (\beta )|\ll (\log Q)^{A}$ for any fixed constant $A>1$ (see Appendix at the end of this paper).
  For larger values of $|\Im (\alpha)|$ and $|\Im (\beta )|$, we have the following conditional result.
\begin{thm}
Let $\eta >0$ be a fixed constant. Assume that the asymptotic formula (\ref{31}) holds for $|\Re (\alpha)|, |\Re (\beta)|\ll 1/\log Q$, $|\Im (\alpha)|, |\Im (\beta)|\ll Q^{\eta}$ with ${\cal E}_{h,k}$ satisfying the condition (\ref{32.5}).  Then, for $T\asymp Q^{\eta}$, we have
\begin{equation}
\label{higher zeros}
{\cal N}_{0}^{'}(T,Q) \geq C(\eta)\; {\cal N}(T, Q),
\end{equation}
where $C(\eta)$ is a function whose values are given in Table 1 below.

\begin{table}[htbp] 
  \caption{The values of $C(\eta)$}
  \label{tab:1}
  \centering
  \begin{tabular}{|c||c|c|c|c|c|c|c|c|c|} \hline
   $\eta$ & 0+ & 0.25 & 0.5 & 0.75 & 1 & 2 & 3 & 4 & 5                                                            \\ \hline
   $C(\eta)$ & 0.6044   & 0.5261 & 0.4615 & 0.4072  & 0.3601 & 0.2200 & 0.1266 & 0.0583  & 0.0048                     \\ \hline
  \end{tabular} 
\end{table}
\end{thm}

The function $C(\eta)$ is decreasing, and is negative when $\eta >5.074$. We will describe the concrete construction of $C(\eta)$ in the proof of Theorem 1.3.


\section{Evaluating the proportion of critical zeros}
In this section, we  see the fundamental method of Levinson and Conrey to detect the critical zeros of Dirichlet $L$-functions. Though the basic argument is almost the same as that of \cite{CIS2}, we recall the method, since the way of taking average over conductors in \cite{CIS2} and ours are slightly different. Let $\chi$ be an even primitive character modulo $q$. Put
\begin{equation}
\label{6}
H_{q}(s)=\left( \frac{q}{\pi} \right)^{s/2-1/4}\Gamma \left( \frac{s}{2} \right).
\end{equation}
Then by (\ref{2}) and (\ref{3}), we have 
\begin{equation}
\label{7}
\eta _{\chi}H_{q}(s)L(s, \chi)=\overline{\eta _{\chi}}H_{q}(1-s)L(1-s, \overline{\chi}),
\end{equation}
where $\varepsilon _{\chi}=\overline{\eta _{\chi}}/\eta _{\chi}$. This $\eta _{\chi}$ satisfies $\eta _{\overline{\chi}}=\overline{\eta _{\chi}}$. Next, define a function $G(s,\chi)$ by 
\begin{equation}
\label{8}
H_{q}(s)G(s, \chi)=\eta_{\chi}\Lambda (s,\chi)+\eta _{\chi}\sum _{k>0, \mathrm{odd}}\frac{g_{k}}{(\log q)^{k}}\Lambda ^{(k)}(s,\chi),
\end{equation}
where $(g_{k})$ is a real sequence supported in the set of positive odd integers and $g_{k}=0$ for all sufficiently large $k$. Since
\begin{equation}
\label{9}
\Lambda ^{(k)}(s,\chi)=(-1)^{k}\varepsilon_{\chi}\Lambda ^{(k)}(1-s, \overline{\chi}),
\end{equation}
by (\ref{8}) we have
\begin{equation}
\label{10}
H_{q}(1-s)G(1-s, \overline{\chi})=\eta _{\chi}\Lambda (s, \chi)-\eta _{\chi}\sum _{k>0, \mathrm{odd}}\frac{g_{k}}{(\log q)^{k}}\Lambda ^{(k)}(s,\chi).
\end{equation}
By (\ref{8}) and (\ref{10}), we have
\begin{equation}
\label{11}
2\eta _{\chi}\Lambda (s,\chi)=H_{q}(s)G(s,\chi)+H_{q}(1-s)G(1-s, \overline{\chi}).
\end{equation}
By (\ref{11}), it follows that $\rho =1/2+i\gamma$ is a zero of $L(s,\chi)$ if and only if  either $\rho$ is a zero of $G(s,\chi)$ or  $G(1/2+i\gamma , \chi)\neq 0$ and $\Re (H_{q}(1/2+i\gamma)G(1/2+i\gamma ,\chi))=0$. Hence in order to obtain a lower bound for the proportion of critical zeros, it suffices to count the number of $\gamma \in [T, 2T]$ satisfying 
\begin{equation}
\label{12}
H_{q}(1/2+i\gamma)G(1/2+i\gamma ,\chi) \equiv \frac{\pi}{2} \quad  (\mathrm{mod}\; \pi ).
\end{equation}
Consequently, 
\begin{equation}
\label{12.5}
N_{0}(T, \chi) \geq \frac{1}{\pi}\Delta _{{\cal C}}\arg \left( H_{q}(s)G(s, \chi) \right)-1,
\end{equation}
where ${\cal C}=\{s=1/2+it \;  |\; T \leq t \leq 2T \}$. By Stirling's approximation of the Gamma function and (\ref{4}), we have
\[
\frac{1}{\pi}\Delta _{{\cal C}}\arg H_{q}(s)=N(T, \chi) +O(\log qT).
\]
Hence 
\begin{equation}
\label{13}
N_{0}(T, \chi) \geq N(T, \chi)+\frac{1}{\pi}\Delta _{{\cal C}}\arg G(s, \chi)+O(\log qT).
\end{equation}
Let ${\cal R}$ be a rectangle with vertices $1/2+iT$, $1/2+2iT$, $A_{0}+iT$ and $A_{0}+2iT$ oriented clockwise, where   $A_{0}>3$ is a sufficiently large fixed constant. Then 
\begin{equation}
\label{13.5}
\Delta _{{\cal C}}\arg G(s, \chi)=\Delta _{{\cal R}}\arg G(s, \chi) +O(\log qT).
\end{equation}
Denote the number of zeros of $G(s, \chi)$ in ${\cal R}$ by $N_{R}(G)$. Then
\begin{equation}
\label{13.75}
\Delta _{{\cal R}}\arg G(s, \chi)=-2\pi N_{R}(G).
\end{equation}
Combining (\ref{13}), (\ref{13.5}) and (\ref{13.75}), we have
\begin{equation}
\label{14}
N_{0}(T, \chi) \geq N(T, \chi)-2N_{{\cal R}}(G)+O(\log qT).
\end{equation} 
For  an entire function  $M(s, \chi)$ and  $\sigma <1/2$, put
\begin{equation}
\label{15}
F(s, \chi)=\eta _{\chi}^{-1}M\left( \frac{1}{2}-\sigma +s, \chi\right)G(s, \chi).
\end{equation}
We denote the number of zeros of $F(s, \chi)$ in ${\cal R}$ by $N_{{\cal R}}(F)$. Let $D$ be a clockwise oriented rectangle with vertices $\sigma +iT$, $\sigma +2iT$, $A_{0}+iT$ and $A_{0}+2iT$.  Then by Littlewood's lemma we have
\begin{equation}
\label{16}
\Re \left( \frac{1}{2\pi i}\int _{\partial D}\log F(s, \chi) ds \right) =\sum _{\rho \in D}\mathrm{dist} (\rho) \geq \left( \frac{1}{2}-\sigma \right)N_{{\cal R}}(F),
\end{equation}
where $\rho$ runs over zeros of $F(s, \chi)$ counted with multiplicity and $\textrm{dist}(\rho)$ is the distance between $\rho$ and the left edge of $D$. The contribution of the integral except for the left edge in (\ref{16}) is $O(\log qT )$, provided that $F(s, \chi)$ is of the form
\begin{equation}
\label{F}
F(s, \chi) =\left( 1+\sum _{n=2}^{\infty}\frac{\lambda (n)}{n^{s}}\right) (1+o(1)), \quad \lambda (n) \ll 1.
\end{equation}
By (\ref{16}) we have
\begin{equation}
\label{18}
N_{{\cal R}}(F) \leq \frac{1}{2\pi (1/2-\sigma )}\left( \int _{T}^{2T}\log |F(\sigma +it, \chi)| dt +O(\log qT) \right).
\end{equation}
By (\ref{14}),  (\ref{18}) and $N_{{\cal R}}(G)\leq N_{{\cal R}}(F)$, we have
\begin{equation}
\label{19}
N_{0}(T, \chi) \geq N(T, \chi)-\frac{1}{\pi (1/2-\sigma )}\int _{T}^{2T}\log |F(\sigma +it, \chi)| dt +O\left(\frac{\log qT}{1/2-\sigma}\right).
\end{equation}
Multiplying both sides of (\ref{19}) by $W(q/Q)$ and taking the sum over even primitive characters and conductors, we have
\begin{equation}
\label{20}
\begin{aligned} {\cal N}_{0}(T, Q) \geq & {\cal N}(T,Q)  -\frac{T}{\pi (1/2-\sigma)}\sum _{q}W\left( \frac{q}{Q} \right)\sum _{\chi \; (\mathrm{mod}\; q)}^{\quad \quad \flat}\frac{1}{T}\int _{T}^{2T}\log |F(\sigma +it, \chi)| dt \\
& +O\left( \frac{Q^{2}\log QT}{1/2-\sigma}\right).
\end{aligned}
\end{equation}
Denote the number of even primitive characters modulo $q$ by $\varphi ^{\flat}(q)$ and put 
\[
\psi (Q)=\sum _{q}W\left( \frac{q}{Q} \right)\varphi ^{\flat}(q).
\]
Due to the convexity of the logarithmic function, we have 
\begin{equation}
\label{21}
\begin{aligned}
& \frac{1}{\psi (Q)}\sum _{q}W\left( \frac{q}{Q} \right)\sum _{\chi \; (\mathrm{mod}\; q)}^{\quad \quad \flat}\frac{1}{T}\int _{T}^{2T}\log |F(\sigma +it, \chi)| dt \\
& \quad \leq \log \left( \frac{1}{\psi (Q)}\sum _{q}W\left( \frac{q}{Q} \right)\sum _{\chi \; (\mathrm{mod}\; q)}^{\quad \quad \flat}\frac{1}{T}\int _{T}^{2T} |F(\sigma +it, \chi)| dt \right).
\end{aligned}
\end{equation}
By Cauchy's inequality, 
\begin{equation}
\label{22}
\begin{aligned}
& \frac{1}{\psi (Q)}\sum _{q}W\left( \frac{q}{Q} \right)\sum _{\chi \; (\mathrm{mod}\; q)}^{\quad \quad \flat}\frac{1}{T}\int _{T}^{2T} |F(\sigma +it, \chi)| dt \\
& \quad \leq \left( \frac{1}{\psi (Q)}\sum _{q}W\left( \frac{q}{Q} \right)\sum _{\chi \; (\mathrm{mod}\; q)}^{\quad \quad \flat}\frac{1}{T}\int _{T}^{2T} |F(\sigma +it, \chi)|^{2} dt \right)^{\frac{1}{2}}.
\end{aligned}
\end{equation}
Combining (\ref{21}) and (\ref{22}), we have 
\begin{equation}
\label{23}
\begin{aligned}
& \sum _{q}W\left( \frac{q}{Q} \right)\sum _{\chi \; (\mathrm{mod}\; q)}^{\quad \quad \flat}\frac{1}{T}\int _{T}^{2T}\log |F(\sigma +it, \chi)| dt \\
& \quad \leq \frac{\psi (Q)}{2}\log \left(\frac{1}{\psi (Q)}\sum _{q}W\left( \frac{q}{Q} \right)\sum _{\chi \; (\mathrm{mod}\; q)}^{\quad \quad \flat}\frac{1}{T}\int _{T}^{2T} |F(\sigma +it, \chi)|^{2} dt   \right).
\end{aligned}
\end{equation}
Therefore, 
\begin{equation}
\label{24}
\begin{aligned}
{\cal N}_{0}(T, Q) \geq &{\cal N}(T,Q) -\frac{T\psi (Q)}{2\pi (1/2-\sigma)}\log \left( \frac{1}{\psi (Q)}\sum _{q}W\left( \frac{q}{Q} \right)\sum _{\chi \; (\mathrm{mod}\; q)}^{\quad \quad \flat}\frac{1}{T}\int _{T}^{2T} |F(\sigma +it, \chi)|^{2} dt   \right) \\
&  +O\left( \frac{Q^{2}\log QT}{1/2-\sigma} \right).
\end{aligned}
\end{equation}
Combining (\ref{24}) with 
\begin{equation}
\label{25}
\psi (Q)\sim \frac{2\pi}{T\log QT}\; {\cal N}(T,Q),
\end{equation}
we have 
\begin{equation}
\label{26}
\begin{aligned}
& {\cal N}_{0}(T,Q) \\
& \quad \geq \left\{ 1-\frac{1}{(1/2-\sigma )\log QT}\log \left(\frac{1}{\psi (Q)}\sum _{q}W\left( \frac{q}{Q} \right)\sum _{\chi \; (\mathrm{mod}\; q)}^{\quad \quad \flat}\frac{1}{T}\int _{T}^{2T} |F(\sigma +it, \chi)|^{2} dt   \right) \right. \\
& \left. \quad +O\left( \frac{1}{(1/2-\sigma)T} \right) \right\}(1+o(1)){\cal N}(Q,T).
\end{aligned}
\end{equation}
For a constant $R>0$, put $\sigma =1/2-R/\log QT$. The error term above is $O(T^{-1}\log QT)$, which is negligible if $T \gg (\log Q)^{2}$. Hence we obtain the following inequality. 
\begin{prop}
Assume that $F(s, \chi)$ is of the form (\ref{F}). Then, for a positive constant $R$,  we have
\begin{equation}
\label{27}
\begin{aligned}
& {\cal N}_{0}(T,Q) \\
& \quad \geq \left\{ 1-\frac{1}{R}\log \left( \frac{1}{\psi (Q)}\sum _{q}W\left( \frac{q}{Q} \right)\sum _{\chi \; (\mathrm{mod}\; q)}^{\quad \quad \flat}\frac{1}{T}\int _{T}^{2T} |F(\sigma +it, \chi)|^{2} dt    \right)   \right\} \\
& \quad \quad  \times (1+o(1)){\cal N}(T,Q)
\end{aligned}
\end{equation}
for $T \gg (\log Q)^{2}$, where  $\sigma =1/2-R/\log QT$.
\end{prop}

Before proceeding to the computation of mollified moment, let us see the condition that (\ref{F}) holds. Put $P(x)=\sum _{l}g_{l}x^{l}$ ($g_{0}:=1$).  Let $A_{0}>1$ be a sufficiently large fixed constant. Assume that $M(s, \chi)$ satisfies $M(A_{0}+it, \chi)\sim 1$ as $A_{0}\to \infty$. Then
\[
F(A_{0}+it, \chi)=\eta _{\chi}^{-1}M\left( \frac{1}{2}-\sigma +A_{0} +it, \chi \right)G(A_{0}+it, \chi)\sim \eta _{\chi}^{-1}G(A_{0}+it, \chi).
\]
Hence $G(s, \chi)$ must satisfy $\eta _{\chi}^{-1}G(A_{0}+it, \chi)\sim 1$. By (\ref{8}) and Stirling's approximation of logarithmic derivatives 
\[
\frac{H_{q}^{(l)}(s)}{H_{q}(s)} \sim \left( \frac{1}{2}\log \frac{qT}{\pi} \right)^{l},
\]
we have
\begin{align*}
\eta _{\chi}^{-1}G(s, \chi)\sim &L(s, \chi) +\sum _{k: \mathrm{odd}}\frac{g_{k}}{(\log q)^{k}}\sum _{l=0}^{k}\binom{k}{l}\left( \frac{1}{2}\log \frac{qT}{\pi} \right)^{l}L^{(k-l)}(s, \chi) \\
=&P \left( \frac{\log \frac{qT}{\pi}}{2\log q}+\frac{d}{ds} \right)L(s, \chi).
\end{align*}
Hence the polynomial $P(x)$ is required to satisfy $P((\log qT/\pi )/(2\log q))\sim 1$. If $T\asymp Q^{\eta}$, $q \asymp Q$, this condition is equivalent to
\begin{equation}
\label{condP}
P\left( \frac{1+\eta}{2} \right)=1.
\end{equation}
In case of $(\log Q)^{2} \leq T \leq (\log Q)^{A}$, the condition 
\begin{equation}
\label{condP'}
P\left( \frac{1}{2} \right)=1
\end{equation}
is required. In this case we use the polynomial $Q(x):=P(1/2-x)$. Hence $Q(x)$ is required to satisfy
\begin{equation}
\label{condQ}
Q(0)=1.
\end{equation}
By the definition of the polynomial $P(x)$, $Q(x)$ must satisfy another condition that 
\begin{equation}
\label{condQ'}
Q^{'}(x)+Q^{'}(1-x)\equiv 0.
\end{equation}


\section{The adaption of Feng's mollifier}
We use a Feng-type mollifier defined by 
\begin{equation}
\label{14.5}
\begin{aligned}
&M(s, \chi) \\
&=\sum _{m\leq X}\frac{\mu (m)\chi (m)}{m^{s}}P_{1}\left( \frac{\log X/m}{\log X} \right) \\
&\quad +\sum _{m\leq X}\frac{\mu (m)\chi (m)}{m^{s}} \left\{ \sum _{p_{1}p_{2}|m}\frac{\log p_{1}\log p_{2}}{\log ^{2}X} P_{2}\left( \frac{\log X/m}{\log X} \right)\right. \\
& \quad  \left. +\sum _{p_{1}p_{2}p_{3}|m}\frac{\log p_{1}\log p_{2}\log p_{3}}{\log ^{3}X} P_{3}\left( \frac{\log X/m}{\log X} \right) +\cdots +\sum _{p_{1}\cdots p_{I}|m}\frac{\log p_{1}\cdots \log p_{I}}{\log ^{I}X}P_{I}\left( \frac{\log X/m}{\log X} \right) \right\}
\end{aligned}
\end{equation}
for $I \geq 3$, $X=Q^{\theta}$, $0<\theta <1$. Here, $P_{i}$ are real polynomials with $P_{1}(0)=0$, $P_{1}(1)=1$, $P_{j}(0)=0$ $(j=2, \ldots ,I)$. We write (\ref{14.5}) by 
\[
M(s, \chi)=\sum _{m \leq X}\frac{\chi (m)c(m)}{m^{s}}.
\]
Then, since $F(\sigma +it, \chi)$ is given by
\[
F(\sigma +it, \chi)=\eta _{\chi}^{-1}M\left( \frac{1}{2}+it, \chi \right)G(\sigma +it, \chi)
\]
with
\[
M\left( \frac{1}{2}+it, \chi \right)=\sum _{m \leq X}\frac{\chi (m)c(m)}{m^{1/2+it}},
\]
\[
G(\sigma +it, \chi)=\frac{\eta _{\chi}}{H_{q}(\sigma +it)}\sum _{l\geq 0} \frac{g_{l}}{(\log q)^{l}}\Lambda ^{(l)}(\sigma +it, \chi)\quad  (g_{0}:=1),
\]
we have
\begin{equation}
\label{28}
\begin{aligned}
\int _{T}^{2T}|F(\sigma +it, \chi)|^{2} dt =&\sum _{h, k \leq X}\sum _{l, m}\frac{g_{l}g_{m}}{(\log q)^{l+m}} \frac{\chi (h)\overline{\chi}(k)c(h)c(k)}{\sqrt{hk}} \\
& \quad \times \int _{T}^{2T}\frac{1}{|H_{q}(\sigma +it )|^{2}}\Lambda ^{(l)}(\sigma +it, \chi)\Lambda ^{(m)}(\sigma -it, \overline{\chi})\left( \frac{k}{h} \right)^{it} dt.
\end{aligned}
\end{equation}
To compute (\ref{28}), put 
\begin{equation}
\label{deltatilde}
\tilde{\Delta} _{\alpha ,\beta }(h,k;Q):=\sum _{q}\frac{W(q/Q)}{|H_{q}(\sigma +it)|^{2}}\sum _{\chi (\mathrm{mod}\; q)}^{\quad \quad \flat} \Lambda (1/2+\alpha ,\chi )\Lambda (1/2+\beta ,\overline{\chi})\chi (h) \overline{\chi}(k).
\end{equation}
Then
\begin{equation}
\label{30}
\begin{aligned}
& \sum _{q}W\left( \frac{q}{Q} \right)\frac{1}{T}\sum _{\chi (\mathrm{mod}\; q)}^{\quad \quad \flat}\int _{T}^{2T}|F(\sigma +it, \chi)|^{2} dt \\
& \quad =\sum _{h,k\leq X}\sum _{l, m}\frac{g_{l}g_{m}}{(\log q)^{l+m}}\frac{c(h)c(k)}{\sqrt{hk}} \frac{1}{T}\int _{T}^{2T}\frac{\partial ^{l+m}}{\partial \alpha ^{l}\partial \beta ^{m}}\tilde{\Delta}_{\alpha ,\beta}(h,k; Q) | _{\alpha =\sigma -1/2+it,\; \beta =\sigma -1/2-it}\left( \frac{k}{h} \right)^{it} dt.
\end{aligned}
\end{equation}
Under the assumption of (\ref{31}), the sum over $h, k$ is
\begin{equation}
\label{32.75}
\begin{aligned}
& \sum _{h, k \leq X}\frac{c(h)c(k)}{\sqrt{hk}}\tilde{\Delta}_{\alpha ,\beta}(h,k;Q)\left( \frac{k}{h} \right)^{it} \\
& \quad = \sum _{q}W\left( \frac{q}{Q} \right)\varphi ^{*}(q)|H_{q}(\sigma +it )|^{-2} \\
& \quad \times \left\{\left( \frac{q}{\pi} \right)^{\frac{\alpha +\beta}{2}}\Gamma \left( \frac{1}{4}+\frac{\alpha}{2} \right) \Gamma \left( \frac{1}{4}+\frac{\beta}{2} \right) \zeta _{q}(1+\alpha +\beta)\sum _{\underset{(hk,q)=1}{h, k\leq X}}\frac{c(h)c(k)(h,k)^{1+\alpha +\beta}}{h^{1+\beta +it}k^{1+\alpha -it}}  \right. \\
&\left.  \quad \quad \quad  + \left( \frac{q}{\pi} \right)^{-\frac{\alpha +\beta}{2}}\Gamma \left( \frac{1}{4}-\frac{\alpha}{2} \right) \Gamma \left( \frac{1}{4}-\frac{\beta}{2} \right) \zeta _{q}(1-\alpha -\beta)\sum _{\underset{(hk,q)=1}{h, k\leq X}}\frac{c(h)c(k)(h,k)^{1-\alpha -\beta}}{h^{1-\alpha +it}k^{1-\beta -it}} \right\}
\end{aligned}
\end{equation}
plus  negligible remainder terms. Put
\[
F(j, w)=\prod _{p|j}\left(1-\frac{1}{p^{w}} \right).
\]
Then
\[
(h,k)^{1+\alpha +\beta}=\sum _{j|(h,k)}j^{1+\alpha +\beta}F(j, 1+\alpha +\beta).
\]
Hence
\begin{equation}
\label{32.875}
\begin{aligned}
\sum _{\underset{(hk,q)=1}{h, k \leq X}}\frac{c(h)c(k)(h,k)^{1+\alpha +\beta}}{h^{1+\beta +it}k^{1+\alpha -it}} =&\sum _{\underset{(hk,q)=1}{h, k \leq X}}\frac{c(h)c(k)}{h^{1+\beta +it}k^{1+\alpha -it}}\sum _{j|(h,k)}j^{1+\alpha+\beta}F(j, 1+\alpha +\beta) \\
=&\sum _{\underset{(j,q)=1}{j \leq X}}\frac{F(j, 1+\alpha +\beta)}{j} \sum _{\underset{(hk,q)=1}{h,k \leq \frac{X}{j}}}\frac{c(jh)c(jk)}{k^{1+\beta +it}h^{1+\alpha -it}}.
\end{aligned}
\end{equation}
By the definition of $c(n)$, it follows that $c(jh)=0$ if $(h,j)>1$. Our first aim in this section is to compute the sum 
\[
\sum _{\underset{(h, jq)=1}{h \leq X/j}}\frac{c(jh)}{h^{\gamma}}
\]
for $j\leq X$, $(j,q)=1$, $|\gamma -1|\ll 1/\log X$. Put
\begin{equation}
\label{Gjq}
G_{j, q}(\gamma)=\sum _{\underset{(n, jq)=1}{n \leq y/j}} \frac{\mu (n)}{n^{\gamma}}  P\left( \frac{\log y/nj}{\log y} \right).
\end{equation}
The following lemma is an analogue of Lemma 10 of \cite{C2}.
\begin{lem}
Let $P$ be a real polynomial with $P(0)=0$. For $y>1$, $(j,q)=1$, $|\gamma -1|\ll 1/\log y$, we have
\[
G_{j,q}(\gamma)=M_{j,q}(\gamma)+O(E_{j,q}),
\]
where 
\[
M_{j,q}(\gamma)=\frac{1}{F(j,q)}\left\{(\gamma -1)P\left( \frac{\log y/j}{\log y} \right)+\frac{1}{\log y}P^{'}\left( \frac{\log y/j}{\log y} \right)   \right\},
\]
\[
E_{j, q}=\left( \frac{\log \log y}{\log y} \right)^{2}\left(1+(\log y)\left( \frac{j}{y} \right)^{\frac{1}{M\log \log y}} \right)F_{1}(jq, 1-2\delta ).
\]
Here, $F_{1}(a,s)=\prod _{p|a}(1+p^{-s})$, $\delta =1/\log \log y$ and $M$ is some positive constant. 
\end{lem}
The proof of the above formula can be obtained by following the argument in \cite{C2}, hence we introduce only the outline of the proof.
\begin{proof}
Put $x=y/j$. By expanding the polynomial  and expressing the sum over $n$ using Cauchy's integral in (\ref{Gjq}), we have
\[
G_{j,q}(\gamma)=\sum _{l\geq 1}\frac{P^{(l)}(0)}{\log ^{l}y}\frac{1}{2\pi i}\int _{2-i\infty}^{2+i\infty}\frac{x^{s-\gamma}}{\zeta (s)F(jq, s)(s-\gamma)^{l+1}} ds.
\]
The integrand has a pole of order $l+1$ at $s=\gamma$. Put $b=1/(M\log \log y)$. Let $R_{0,l}(\gamma)$ be the residue from the pole at $s=\gamma$,  $R_{1,l}(\gamma)$ be the integral on $s=1+it$, $-\infty <t\leq  -(\log y)^{10}$, $R_{2, l}(\gamma)$ be the integral on $s=\sigma -i(\log y)^{10}$, $1-b \leq \sigma \leq 1$, $R_{3, l}(\gamma)$ be the integral on $s=1-b+it$, $-(\log y)^{10}\leq t \leq (\log y)^{10}$, and $R_{4,l}(\gamma)$ and $R_{5,l}(\gamma)$ be the integrals on paths conjugate to the paths of $R_{2,l}(\gamma)$, $R_{1,l}(\gamma)$, respectively. We choose $M>0$ sufficiently large so that $\zeta (s)$ doesn't vanish in the right side of  $R_{1,l}(\gamma )+\ldots +R_{5,l}(\gamma )$. Then, by moving the path of integration to $R_{1,l}(\gamma )+\ldots +R_{5,l}(\gamma )$, we have 
\begin{equation}
\label{A}
G_{j,q}(\gamma)=\sum _{l \geq 1}\frac{P^{(l)}(0)}{\log ^{l}y}\frac{1}{2\pi i}\left( 2\pi i R_{0,l}(\gamma )+\sum _{m=1}^{5}R_{m,l}(\gamma) \right).
\end{equation}
$R_{1, l}(\gamma), \ldots ,R_{5,l}(\gamma)$ are obtained by replacing $\beta$ with $\gamma$, $F(j,s)$ with $F(jq, s)$ in the proof of Lemma 10 of \cite{C2}. Hence by the argument of \cite{C2}, we have 
\[
R_{1, l}(\gamma), \;  R_{5,l}(\gamma) \ll F_{1}(jq, 1-2\delta )(\log y)^{-9},
\]
\[
R_{2,l}(\gamma), \; R_{4,l}(\gamma) \ll F_{1}(jq, 1-2\delta )(\log y)^{-20},
\]
\[
R_{3,l}(\gamma)\ll F_{1}(jq, 1-2\delta )(\log \log y)^{l+1}\left( \frac{j}{y} \right)^{b}.
\]
Hence the contribution of these integrals to (\ref{A}) is at most
\begin{equation}
\label{B}
\frac{(\log \log y)^{2}}{(\log y)^{2}}\left( 1+ (\log y)\left( \frac{j}{y} \right)^{b} \right)F_{1}(jq, 1-2\delta ).
\end{equation}
On the other hand, $R_{0,l}(\gamma)$ is expanded by
\begin{equation}
\label{C}
R_{0, l}(\gamma)=\frac{1}{l!}\sum _{k=0}^{l}\binom{l}{k}(\log x)^{l-k}Z^{(k)}(\gamma),
\end{equation}
where 
\[
Z(s)=\frac{1}{\zeta (s)F(jq, s)}.
\]
The function $Z(\gamma)$ and its derivatives are expanded by
\[
Z(\gamma)=\frac{\gamma -1 +O(|\gamma -1|^{2})}{F(jq, \gamma)}, \quad 
Z^{'}(\gamma)=\frac{1}{F(jq, \gamma)}+O(F_{1}(jq, 1-2\delta )|\gamma -1|\log \log y),
\]
and derivatives satisfy the bound 
\[
Z^{(k)}(\gamma) \ll F_{1}(jq, 1-2\delta)(\log \log y)^{k-1}.
\]
Hence the contributions of the terms with $k\geq 2$ in (\ref{C}) and the error terms in $Z(\gamma)$ and $Z^{'}(\gamma)$ in the above formulas are bounded by (\ref{B}). Thus we have
\begin{equation}
\label{D}
\begin{aligned}
\sum _{l \geq 1}\frac{P^{(l)}(0)}{\log ^{l}y}R_{0, l}(\gamma) &=\frac{1}{F(jq, \gamma)}\left\{(\gamma -1)P\left( \frac{\log y/j}{\log y} \right)+\frac{1}{\log y}P^{'} \left( \frac{\log y/j}{\log y} \right)  \right\} \\
& \quad +O\left( \frac{(\log \log y)^{2}}{(\log y)^{2}}\left( 1+ (\log y)\left( \frac{j}{y} \right)^{b} \right)F_{1}(jq, 1-2\delta ) \right).
\end{aligned}
\end{equation}
Combining (\ref{D}) with the estimates for $R_{m,l}(\gamma)$ $(m=1, \ldots ,5)$, we obtain the result.   
\end{proof}
By the definition of $c(n)$, we have
\begin{align*}
\sum _{\underset{(h, jq)=1}{h \leq X/j}}\frac{c(jh)}{h^{\gamma}} &=\mu (j) \sum _{\underset{(h, jq)=1}{h \leq X/j}}\frac{\mu (h)}{h^{\gamma}}  \left\{ P_{1}\left( \frac{\log X/jh}{\log X} \right)+\sum _{p_{1}p_{2}|jh}\frac{\log p_{1}\log p_{2}}{\log ^{2}X} P_{2}\left( \frac{\log X/jh}{\log X} \right) \right. \\
& \left. \quad \quad \quad \quad \quad \quad  +\ldots +\sum _{p_{1}\cdots p_{I}|jh}\frac{\log p_{1}\cdots \log p_{I}}{\log ^{I}X} P_{I}\left( \frac{\log X/jh}{\log X} \right) \right\}.
\end{align*}
The left hand side is obtained by replacing $\alpha$ with $\gamma -1$, letting $y=y_{1}=X$ and adding the condition $(h,q)=1$ in the definition of $j^{1+\alpha}E_{\alpha}(j)$ in \cite{F}, (2.12). The sum
\[
\sum _{\underset{(h, jq)=1}{h \leq X/j}}\frac{\mu (h)}{h^{\gamma}}P_{1}\left( \frac{\log X/jh}{\log X} \right)
\]
has been obtained by Lemma 3.1. The contributions of the other terms  are obtained by inductive arguments in Feng's paper \cite{F}, combining with Lemma 3.1. The only difference between Feng's argument and ours is that in our situation the sum has another restriction that $(h,q)=1$, which doesn't cause any significant discrepancy in the arguments.  Similar to the proof of Lemma 3.1, this difference is modified by replacing $F(j, 1+\alpha)$, $F_{1}(j, 1+\alpha)$ with $F(jq, 1+\alpha)$, $F_{1}(jq, 1+\alpha)$, respectively. Hence we arrive at the following intermediate lemma. 
\begin{lem}[\cite{F}, (3.27)]
We have
\begin{equation}
\label{ch}
\begin{aligned}
& \sum _{\underset{(h, jq)=1}{h \leq X/j}}\frac{c(jh)}{h^{\gamma}} \\
&\quad =\frac{\mu (j)}{F(jq, \gamma)} \left\{ G_{0}(\gamma -1, j)+G_{1}(\gamma -1, j)\sum _{p_{1}|j}\log p_{1} \right. \\
& \left. \quad \quad  +G_{2}(\gamma -1, j)\sum _{p_{1}p_{2}|j}\log p_{1}\log p_{2}+\cdots +G_{I}(\gamma -1,j)\sum _{p_{1}\cdots p_{I}|j}\log p_{1}\cdots \log p_{I} \right\} \\
& \quad \quad  +O\left( \frac{\mu (j)F_{1}(jq, 1-2\delta )(\log \log X)^{3}}{j \log ^{2}X} \right) +O\left( \frac{\mu (j)F_{1}(jq, 1-2\delta )(\log \log X)^{2}}{j \log X} \left( \frac{j}{X} \right)^{d} \right).
\end{aligned}
\end{equation}
Here, $\delta =1/\log \log X$, $d=1/(M\log \log X)$ and $G_{i}(\alpha ,X)$ $(i=0, \ldots ,I)$ are defined by 
\begin{align*}
G_{0}(\alpha ,j)  =&\alpha P_{1}\left( \frac{\log X/j}{\log X} \right) +\frac{1}{\log X}P_{1}^{'}\left( \frac{\log X/j}{\log X} \right) \\
& \quad  +\sum _{l=2}^{I} \frac{(-1)^{l}}{(l-2)!\log ^{l}X}\int _{1}^{\frac{X}{j}}P_{l}\left( \frac{\log X/jx}{\log X} \right) \frac{\log ^{l-2}x}{x^{1+\alpha}} dx, 
\end{align*}
\begin{align*}
G_{1}(\alpha ,j) =&-\frac{2}{\log ^{2}X}P_{2}\left( \frac{\log X/j}{\log X} \right) \\
& \quad  +\sum _{l=3}^{I}\binom{l}{1} \frac{(-1)^{l-1}}{(l-3)!\log ^{l}X}\int _{1}^{\frac{X}{j}}P_{l}\left( \frac{\log X/jx}{\log X} \right) \frac{\log ^{l-3}x}{x^{1+\alpha}} dx,
\end{align*}
\begin{align*}
G_{m}(\alpha ,j) =&\frac{1}{\log ^{m}X}\left( \alpha P_{m}\left( \frac{\log X/j}{\log X} \right) +\frac{1}{\log X}P_{m}^{'}\left( \frac{\log X/j}{\log X} \right) \right) \\
& \quad  - \binom{m+1}{m}\frac{1}{\log ^{m+1}X}P_{m+1}\left( \frac{\log X/j}{\log X} \right) \\
& \quad  +\sum _{l=m+2}^{I}\binom{l}{m}\frac{(-1)^{l-m}}{(l-m-2)!\log ^{l}X}\int _{1}^{\frac{X}{j}}P_{l}\left( \frac{\log X/jx}{\log X} \right) \frac{\log ^{l-m-2}x}{x^{1+\alpha}} dx
\end{align*}
$(2\leq m \leq I-2)$, 
\begin{align*}
G_{I-1}(\alpha ,j)=\frac{\alpha P_{I-1}\left( \frac{\log X/j}{\log X} \right) +\frac{1}{\log X}P_{I-1}^{'}\left( \frac{\log X/j}{\log X} \right)}{\log ^{I-1}X}-\frac{I}{\log ^{I}X}P_{I}\left( \frac{\log X/j}{\log X} \right),
\end{align*}
and 
\begin{align*}
G_{I}(\alpha ,j)=\frac{1}{\log ^{I}X}\left( \alpha P_{I}\left( \frac{\log X/j}{\log X} \right) +\frac{1}{\log X}P_{I}^{'}\left( \frac{\log X/j}{\log X} \right) \right).
\end{align*}
\end{lem}
Combining (\ref{ch}) and Feng's estimates on error terms, it follows that 
\begin{equation}
\label{33}
\begin{aligned}
& \sum _{\underset{(hk, q)=1}{h, k\leq X}}\frac{c(h)c(k)(h,k)^{1+\alpha +\beta}}{h^{1+\beta +it}k^{1+\alpha -it}} \\
& \quad =\sum _{\underset{(j,q)=1}{j \leq X}}\frac{F(j, 1+\alpha +\beta )}{j}\sum _{\underset{(hk,q)=1}{h, k \leq X/j}}\frac{c(jh)c(jk)}{h^{1+\beta +it}k^{1+\alpha -it}} \\
& \quad =\sum _{\underset{(j,q)=1}{j \leq X}} \frac{\mu ^{2}(j)F(j, 1+\alpha +\beta )}{jF(jq, 1+\beta +it)F(jq, 1+\alpha -it)} \\
& \quad \quad \times \left\{ G_{0}(\beta +it,j)+G_{1}(\beta +it, j)\sum _{p_{1}|j}\log p_{1} \right. \\
& \left. \quad \quad \quad  +G_{2}(\beta +it,j)\sum _{p_{1}p_{2}|j}\log p_{1}\log p_{2}+\ldots +G_{I}(\beta +it,j)\sum _{p_{1}\cdots p_{I}|j}\log p_{1}\cdots \log p_{I} \right\} \\
& \quad \quad \times \left\{ G_{0}(\alpha -it,j)+G_{1}(\alpha -it, j)\sum _{p_{1}|j}\log p_{1} \right. \\
& \left. \quad \quad \quad +G_{2}(\alpha -it,j)\sum _{p_{1}p_{2}|j}\log p_{1}\log p_{2}+\ldots +G_{I}(\alpha -it,j)\sum _{p_{1}\cdots p_{I}|j}\log p_{1}\cdots \log p_{I} \right\} \\
& \quad \quad +E,
\end{aligned}
\end{equation}
where $E$ is obtained by replacing $F(j, 1+\alpha +\beta)$, $F_{1}(j, 1-2\delta)$, $\alpha$, $\beta$, $y$ with $F(jq, 1+\alpha +\beta)$, $F_{1}(jq, 1-2\delta)$, $\alpha -it$, $\beta +it$, $X$, respectively in the definition of $U_{2}+\ldots +U_{8}+U_{2}^{'}+\ldots +U_{8}^{'}$ in (4.1) of \cite{F}. By (4.21) of \cite{F}, we have 
\begin{equation}
\label{34}
\begin{aligned}
E=& O\left(F_{1}^{3}(q, 1-2\delta)\frac{(\log \log X)^{5}}{\log ^{2}X}  \right) \\
=&O\left(\prod _{p|q}\left(1+\frac{1}{p^{1-2\delta}} \right)^{3} \frac{(\log \log X)^{5}}{\log ^{2}X}  \right),
\end{aligned}
\end{equation}
where $\delta =1/\log \log X$. 
\begin{lem}
The contribution of $E$ to (\ref{32.75}) is $O(Q^{2}\ (\log \log X)^{5}/(\alpha +\beta )\log ^{2}X)$ uniformly for $|\alpha -it|, |\beta +it| \ll 1/ \log X$. 
\end{lem}
\begin{proof}
Recall that $\sigma =1/2-R/\log QT$. Then,
\[
\frac{\Gamma \left( \frac{1}{4}+\frac{\alpha}{2} \right)\Gamma \left( \frac{1}{4}+\frac{\beta}{2} \right)}{|H_{q}(\sigma +it )|^{2}}\sim \left( \frac{q}{\pi} \right)^{1-2\sigma}\frac{\Gamma \left( \frac{1}{4}+\frac{it}{2} \right)\Gamma \left(   \frac{1}{4}-\frac{it}{2}  \right)}{\Gamma \left( \frac{\sigma +it}{2} \right)\Gamma \left( \frac{\sigma -it}{2} \right)} \ll 1.
\]
Moreover, 
\begin{align*}
\zeta _{q}(1+\alpha +\beta )&=\zeta (1+\alpha +\beta )\prod _{p|q}\left( 1-\frac{1}{p^{1+\alpha +\beta }} \right) \\
&\ll \frac{1}{\alpha +\beta} \prod _{p|q}\left( 1+\frac{1}{p^{1-2\delta}} \right), 
\end{align*}
and
\[
\left( \frac{q}{\pi} \right)^{\frac{\alpha +\beta}{2}}\ll 1, \quad \varphi ^{*}(q)\ll Q, \quad W\left( \frac{q}{Q} \right)\ll 1.
\]
Hence the contribution of $E$ to (\ref{32.75}) is at most
\begin{align*}
& \sum _{q \in Q \mathrm{supp}W}\frac{Q}{\alpha +\beta} \prod _{p|q}\left( 1+\frac{1}{p^{1-2\delta}} \right)^{4}\frac{(\log \log X)^{5}}{\log ^{2}X} \\
& \quad  \ll \frac{Q(\log \log X)^{5}}{(\alpha +\beta )\log ^{2}X} \sum _{q \in Q \mathrm{supp}W} \prod _{p|q}\left( 1+\frac{1}{p^{1-2\delta}} \right)^{4} \\
& \quad  \ll \frac{Q(\log \log X)^{5}}{(\alpha +\beta )\log ^{2}X} \sum _{q \in Q \mathrm{supp}W}\prod _{p|q}\left( 1+\frac{15}{p^{1-2\delta}} \right) \\
& \quad  \ll \frac{Q(\log \log X)^{5}}{(\alpha +\beta )\log ^{2}X} \sum _{q \in Q \mathrm{supp}W}\sum  _{n|q}\frac{d_{15}(n)}{n^{1-2\delta}} \\
& \quad  \ll \frac{Q(\log \log X)^{5}}{(\alpha +\beta )\log ^{2}X} \sum _{n}\frac{d_{15}(n)}{n^{1-2\delta}}\sum _{\underset{n|q}{q\in Q \mathrm{supp}W}}1 \\
& \quad  \ll \frac{Q^{2}(\log \log X)^{5}}{(\alpha +\beta )\log ^{2}X}.
\end{align*}
\end{proof}
We compute the derivatives by using Cauchy's integral formula. In the integration, $\alpha$ (resp. $\beta$) moves on the circle of center $\sigma -1/2 +it$ (resp. $\sigma -1/2-it$) and radius $c_{1}/\log X$ (resp. $c_{2}/\log X$), where $c_{1}$ and $c_{2}$ are positive constants for which $\alpha +\beta \gg 1/\log X$ on these circles. Then, the contribution of the terms arising from derivatives of $E$ to (\ref{30}) is at most $O(Q^{2}(\log \log X)^{5}/\log X)=o(Q^{2})$.

To compute the sum over $j$ in (\ref{33}), we use the following lemmas. 
\begin{lem}[\cite{L}, Lemma 3.9]
For large square-free $j$, 
\[
\sum _{p|j}\frac{\log p}{p}=O(\log \log j).
\]
\end{lem}
\begin{lem}[\cite{L}, Lemma 3.11 or \cite{F}, Lemma 6]
Let $N$ be a positive integer and put
\[
J(x)=\sum _{\underset{(n,N)=1}{n \leq x}}\frac{\mu ^{2}(n)}{n}\prod _{p|n}(1+f(p)),
\]
where 
\[
f(p)=O \left( \frac{1}{p^{c}} \right)
\]
for some $c>0$. Then, 
\[
J(x)=\prod _{p|N}\left(1-\frac{1}{p} \right)\prod _{(p,N)=1}\left(1-\frac{1}{p^{2}} \right)\left( 1+\frac{f(p)}{p+1} \right)\log x +O(\log \log (N+1)).
\]
The implied constant is independent of $N$ and $x$.
\end{lem}

\begin{lem}[\cite{F}, Lemma 7]
Let $m$ be a positive integer, $\alpha$  a complex number,  $f$  a continuous function and  $D \geq 1$. Then 
\begin{equation}
\label{integral}
\begin{aligned}
&\int _{1}^{D}\frac{1}{x_{1}^{1+\alpha}} \int _{1}^{D/x_{1}}\frac{1}{x_{2}^{1+\alpha}} \cdots \int _{1}^{D/x_{1}x_{2}\cdots x_{m-2}}\frac{1}{x_{m-1}^{1+\alpha}}\int _{1}^{D/ x_{1}x_{2}\cdots x_{m-1}}\frac{f(x_{1}x_{2}\cdots x_{m})}{x_{m}^{1+\alpha}} dx_{m} \cdots dx_{2}dx_{1} \\
& \quad =\int _{1}^{D}\frac{f(x)\log ^{m-1}x}{(m-1)!x^{1+\alpha}} dx.
\end{aligned}
\end{equation}
\end{lem}
\begin{lem}[\cite{F}, Lemma 8]
For positive integers $m_{1}, m_{2}$ and square-free $j$, 
\begin{equation}
\label{psum}
\begin{aligned}
& \sum _{p_{1}\cdots p_{m_{1}}|j}\log p_{1} \cdots \log p_{m_{1}}\sum _{q_{1} \cdots q_{m_{2}}|j}\log q_{1} \cdots \log q_{m_{2}} \\
& \quad =\sum _{k=0}^{\min \{m_{1}, m_{2} \}}k! \binom{m_{1}}{k} \binom{m_{2}}{k}\sum _{p_{1}\cdots p_{m_{1}+m_{2}-k}|j}\log ^{2}p_{1} \cdots \log ^{2}p_{k}\log p_{k+1} \cdots \log p_{m_{1}+m_{2}-k}.
\end{aligned}
\end{equation}
\end{lem}
\begin{lem}[\cite{F}, (4.16)]
We have
\begin{equation}
\label{Ges}
G_{m}(\alpha , u)\ll  \frac{1}{\log ^{m+1}X}
\end{equation}
uniformly for  $u \leq X$, $0\leq m \leq I$, $\alpha \ll 1/\log X$. 
\end{lem}
By this lemma we also have 
\begin{equation}
\label{Gdash}
\frac{\partial}{\partial u}G_{m}(\alpha -it, u) \ll \frac{1}{u \log ^{m+2}X}.
\end{equation}
In our situation we may assume that $\alpha -it, \beta +it \ll 1/\log X$. We use the above estimates by replacing $\alpha$, $\beta$ with $\alpha -it$, $\beta +it$, respectively. We need to compute the contribution of the main terms of (\ref{33}).

\begin{lem}
For $|\alpha -it|, |\beta +it| \ll 1/ \log X$, we have
\begin{equation}
\label{38}
\begin{aligned}
& \sum _{\underset{(j,q)=1}{j \leq X}}\frac{\mu ^{2}(j)F(j, 1+\alpha +\beta )}{jF(j, 1+\alpha -it)F(j, 1+\beta +it)}G_{0}(\alpha -it, j)G_{0}(\beta +it, j) \\
& \quad =\frac{F(q,1)}{\log X} \int _{0}^{1}V_{0}(\alpha -it, u)V_{0}(\beta +it ,u) du +O\left( \frac{\log (\log q \log X)}{\log ^{2} X} \right).
\end{aligned}
\end{equation}
Here, $F(q,1)=\prod _{p|q}(1-1/p)$ and 
\begin{equation}
\label{V0}
\begin{aligned}
V_{0}(\alpha ,u)&=V_{0}(\alpha ,u ;X) \\
&=\alpha \log XP_{1}(1-u)+P_{1}^{'}(1-u)+\sum _{l=2}^{I}\frac{(-1)^{l}}{(l-2)!}\int _{0}^{1-u}P_{l}(1-u-\mu )\mu ^{l-2}e^{-\alpha \mu \log X} d\mu .
\end{aligned}
\end{equation}
\end{lem}
\begin{proof}
First, 
\begin{align*}
\frac{F(j, 1+\alpha +\beta )}{F(j, 1+\alpha -it)F(j, 1+\beta +it)} &=\prod _{p|j}\frac{1-\frac{1}{p^{1+\alpha +\beta }}}{\left( 1-\frac{1}{p^{1+\alpha -it}} \right) \left( 1-\frac{1}{p^{1+\beta +it }} \right)} \\
&= \prod _{p|j}\left( 1+\frac{p^{1+\alpha -it}+p^{1+\beta +it}-1-p}{(p^{1+\alpha -it}-1)(p^{1+\beta +it}-1)} \right).
\end{align*}
Hence we adapt Lemma 3.5 with 
\begin{equation}
\label{fp}
f(p)=\frac{p^{1+\alpha -it}+p^{1+\beta +it}-1-p}{(p^{1+\alpha -it}-1)(p^{1+\beta +it}-1)}.
\end{equation}
By Lemma 3.5, it follows that 
\[
\sum _{\underset{(j, q)=1}{j \leq u}}\frac{\mu ^{2}(j)F(j, 1+\alpha +\beta )}{jF(j, 1+\alpha -it)F(j, 1+\beta +it)}=\prod _{p|q}\left(1-\frac{1}{p} \right)\prod _{p\nmid q} \left(1-\frac{1}{p^{2}} \right)\left(1+\frac{f(p)}{p+1} \right)\log u +E_{q}(u),
\]
where $E_{q}(u)\ll \log \log q$ uniformly for $u$. By abelian summation, Lemma 3.5 and estimates (\ref{Ges}), (\ref{Gdash}), we have
\begin{equation}
\label{35}
\begin{aligned}
& \sum _{\underset{(j,q)=1}{j \leq X}}\frac{\mu ^{2}(j)F(j, 1+\alpha +\beta )}{jF(j, 1+\alpha -it)F(j, 1+\beta +it)}G_{0}(\alpha -it, j)G_{0}(\beta +it, j)  \\ 
& \quad =\int _{1}^{X}G_{0}(\alpha -it, u)G_{0}(\beta +it ,u)  d \left( \prod _{p|q}\left(1-\frac{1}{p} \right)\prod _{p\nmid q} \left(1-\frac{1}{p^{2}} \right)\left(1+\frac{f(p)}{p+1} \right)\log u +E_{q}(u)  \right) \\
& \quad = \prod _{p|q}\left(1-\frac{1}{p} \right)\prod _{p\nmid q} \left(1-\frac{1}{p^{2}} \right)\left(1+\frac{f(p)}{p+1} \right) \int _{1}^{X} \frac{G_{0}(\alpha -it, u)G_{0}(\beta +it, u)}{u} du \\
& \quad \quad  +O\left( \frac{\log \log q}{\log ^{2}X} \right).
\end{aligned}
\end{equation}
Put 
\[
Y_{q}(\alpha ,\beta ;t)= \prod _{p|q}\left(1-\frac{1}{p} \right)\prod _{p\nmid q} \left(1-\frac{1}{p^{2}} \right)\left(1+\frac{f(p)}{p+1} \right).
\]
Then 
\begin{equation}
\label{Y}
\begin{aligned}
Y_{q}(\alpha ,\beta ;t)&=Y_{q}(it, -it, t)+O(|\alpha -it|) +O(|\beta +it|) +O\left( \frac{1}{\log X} \right) \\
&= F(q,1)+O\left( \frac{1}{\log X} \right)
\end{aligned}
\end{equation}
uniformly for $|\alpha -it|, |\beta +it| \ll 1/ \log X$. Hence (\ref{35}) equals 
\begin{equation}
\label{36}
F(q,1)\int _{1}^{X} \frac{G_{0}(\alpha -it, u)G_{0}(\beta +it, u)}{u} du +O\left( \frac{\log (\log q \log X)}{\log ^{2}X} \right).
\end{equation}
The integral in (\ref{36}) is computed by \cite{F}, which is obtained by multiplying $\log ^{-1} X$ by the integral of (\ref{38}). 
\end{proof}

\begin{lem}
For $m=1, \ldots ,I$, we have 
\begin{equation}
\label{48}
\begin{aligned}
& \sum _{\underset{(j,q)=1}{j \leq X}}\frac{\mu ^{2}(j)F(j, 1+\alpha +\beta )}{jF(j, 1+\alpha -it)F(j, 1+\beta +it)}G_{0}(\alpha -it, j)G_{m}(\beta +it, j)\sum _{p_{1}\cdots p_{m}|j}\log p_{1}\cdots \log p_{m} \\
& \quad =\frac{F(q, 1)}{\log X}\int _{0}^{1}\frac{V_{0}(\alpha -it, u)V_{m}(\beta +it, u)u^{m}}{m!} du +O\left( \frac{\log (\log q \log X)}{\log ^{2}X} \right)
\end{aligned}
\end{equation}
uniformly for $|\alpha -it|, |\beta +it| \ll 1/\log X$. Here, $V_{0}(\alpha, u)$ is defined by (\ref{V0}), and 
\begin{equation}
\label{V1}
\begin{aligned}
V_{1}(\alpha ,u)&=V_{1}(\alpha ,u ;X) \\
&=-2P_{2}(1-u)+\sum _{l=3}^{I}\binom{l}{1}\frac{(-1)^{l-1}}{(l-3)!}\int _{0}^{1-u}P_{l}(1-u-\mu )\mu ^{l-3}e^{-\alpha \mu \log X} d\mu ,
\end{aligned}
\end{equation}
\begin{equation}
\label{Vm}
\begin{aligned}
V_{m}(\alpha ,u)&=V_{m}(\alpha ,u ;X) \\
&=\alpha \log X P_{m}(1-u)+P_{m}^{'}(1-u)-\binom{m+1}{m}P_{m+1}(1-u) \\
& \quad +\sum _{l=m+2}^{I}\binom{l}{m}\frac{(-1)^{l-m}}{(l-m-2)!}\int _{0}^{1-u}P_{l}(1-u-\mu )\mu ^{l-m-2}e^{-\alpha \mu \log X} d\mu \quad (2\leq m \leq I-2) ,
\end{aligned}
\end{equation}
\begin{equation}
\label{VI-1}
\begin{aligned}
V_{I-1}(\alpha ,u)&=V_{I-1}(\alpha ,u ;X) \\
&=\alpha \log XP_{I-1}(1-u)+P_{I-1}^{'}(1-u)-IP_{I}(1-u),
\end{aligned}
\end{equation}
\begin{equation}
\label{VI}
\begin{aligned}
V_{I}(\alpha ,u)&=V_{I}(\alpha ,u; X) =\alpha \log X P_{I}(1-u)+P_{I}^{'}(1-u)
\end{aligned}
\end{equation}
and $F(q,1)=\prod _{p|q}(1-1/p)$.
\end{lem}
\begin{proof}
Changing the order of summation, we have
\begin{equation}
\label{40}
\begin{aligned}
& \sum _{\underset{(j,q)=1}{j \leq X}}\frac{\mu ^{2}(j)F(j, 1+\alpha +\beta )}{jF(j, 1+\alpha -it)F(j, 1+\beta +it)}G_{0}(\alpha -it, j)G_{m}(\beta +it, j)\sum _{p_{1}\cdots p_{m}|j}\log p_{1}\cdots \log p_{m} \\
& \quad =\sum _{\underset{(p_{1}\cdots p_{m}, q)=1}{p _{1}, \ldots ,p_{m}\leq X}}\frac{\mu ^{2}(p_{1}\cdots p_{m})F(p_{1}\cdots p_{m}, 1+\alpha +\beta )\log p_{1}\cdots \log p_{m}}{p_{1}\cdots p_{m}F(p_{1}\cdots p_{m}, 1+\alpha -it)F(p_{1}\cdots p_{m}, 1+\beta +it)} \\
& \quad \times \sum _{\underset{(j_{0}, qp_{1}\cdots p_{m})=1}{j_{0} \leq X/p_{1}\cdots p_{m}}}\frac{\mu ^{2}(j_{0})F(j_{0}, 1+\alpha +\beta )}{j_{0}F(j_{0}, 1+\alpha -it)F(j_{0}, 1+\beta +it)}G_{0}(\alpha -it, p_{1}\cdots p_{m})G_{m}(\beta +it, p_{1}\cdots p_{m}).
\end{aligned}
\end{equation}
By abelian summation and Lemma 3.5, we have
\begin{equation}
\label{40.5}
\begin{aligned}
&  \sum _{\underset{(j_{0}, qp_{1}\cdots p_{m})=1}{j_{0} \leq X/p_{1}\cdots p_{m}}}\frac{\mu ^{2}(j_{0})F(j_{0}, 1+\alpha +\beta )}{j_{0}F(j_{0}, 1+\alpha -it)F(j_{0}, 1+\beta +it)}G_{0}(\alpha -it, p_{1}\cdots p_{m})G_{m}(\beta +it, p_{1}\cdots p_{m}) \\
& \quad =\int _{1}^{X/p_{1}\cdots p_{m}}G_{0}(\alpha -it, up_{1}\cdots p_{m})G_{m}(\beta +it, up_{1}\cdots p_{m}) \\
& \quad \quad \times d \left( \prod _{p|qp_{1}\cdots p_{m}}\left(1-\frac{1}{p} \right)\prod _{(p, qp_{1}\cdots p_{m})=1}\left( 1-\frac{1}{p^{2}} \right)\left(1+\frac{f(p)}{p+1} \right) \log u +E_{qp_{1}\cdots p_{m}}(u) \right),
\end{aligned}
\end{equation}
where $f(p)$ is defined by (\ref{fp}). By  estimates (\ref{Ges}), (\ref{Gdash}), this equals
\begin{equation}
\label{41}
\begin{aligned}
&Y_{q}(\alpha ,\beta ; t)\prod _{r=1}^{m}\left( 1+\frac{1}{p_{r}} \right)^{-1}\left( 1+\frac{f(p_{r})}{p_{r}+1} \right)^{-1}\\
& \quad \quad \times   \int _{1}^{X/p_{1}\cdots p_{m}}\frac{G_{0}(\alpha -it, up_{1}\cdots p_{m})G_{m}(\beta +it, up_{1}\cdots p_{m})}{u} du   +O\left( \frac{\log \log qX}{\log ^{m+2}X} \right).
\end{aligned}
\end{equation}
Substituting (\ref{41}) into (\ref{40}), we have
\begin{equation}
\label{41.5}
\begin{aligned}
& \sum _{\underset{(j,q)=1}{j \leq X}}\frac{\mu ^{2}(j)F(j, 1+\alpha +\beta )}{jF(j, 1+\alpha -it)F(j, 1+\beta +it)}G_{0}(\alpha -it, j)G_{m}(\beta +it, j)\sum _{p_{1}\cdots p_{m}|j}\log p_{1}\cdots \log p_{m} \\
& \quad =Y_{q}(\alpha ,\beta ; t)\sum _{\underset{(p_{1}\cdots p_{m}, q)=1}{p_{1} \cdots p_{m} \leq X}}\mu ^{2}(p_{1}\cdots p_{m})\prod _{r=1}^{m}\frac{(p_{r}^{1+\alpha +\beta}-1)\log p_{r}}{(p_{r}^{1+\alpha -it}-1)(p_{r}^{1+\beta +it}-1)+(p_{r}^{1+\alpha +\beta}-1)} \\
& \quad \quad  \times \int _{1}^{X/p_{1}\cdots p_{m}}\frac{G_{0}(\alpha -it, up_{1}\cdots p_{m})G_{m}(\beta +it, up_{1}\cdots p_{m})}{u} du  \\
& \quad +O\left( \frac{\log \log qX}{\log ^{m+2}X}\sum _{p_{1} \cdots p_{m} \leq X}\prod _{r=1}^{m}\frac{\left| 1 -\frac{1}{p_{r}^{1+\alpha +\beta }} \right|\log p_{r}}{\left| \left(1-\frac{1}{p_{r}^{1+\alpha -it}} \right)\left(1-\frac{1}{p_{r}^{1+\beta +it}} \right) \right| p_{r}} \right).
\end{aligned}
\end{equation}
Denote the right hand side by $H_{1}+O(H_{2})$. Since 
\[
\frac{\left| 1 -\frac{1}{p_{r}^{1+\alpha +\beta }} \right|\log p_{r}}{\left| \left(1-\frac{1}{p_{r}^{1+\alpha -it}} \right)\left(1-\frac{1}{p_{r}^{1+\beta +it}} \right) \right| p_{r}} \ll \frac{\log p_{r}}{p_{r}},
\]
we have
\begin{align*}
\sum _{p_{1}\cdots p_{m} \leq X}\prod _{r=1}^{m}\frac{\left| 1 -\frac{1}{p_{r}^{1+\alpha +\beta }} \right|\log p_{r}}{\left| \left(1-\frac{1}{p_{r}^{1+\alpha -it}} \right)\left(1-\frac{1}{p_{r}^{1+\beta +it}} \right) \right| p_{r}} \ll \sum _{p_{1}, \ldots p_{m}\leq X}\prod _{r=1}^{m}\frac{\log p_{r}}{p_{r}} &\ll \left( \sum _{p\leq X}\frac{\log p}{p} \right)^{m}\\
& \ll \log ^{m}X.
\end{align*}
Hence 
\begin{equation}
\label{42}
H_{2}\ll \frac{\log \log qX}{\log ^{2}X}.
\end{equation}
Next, we compute $H_{1}$. Since 
\begin{equation}
\label{43}
\frac{(p^{1+\alpha +\beta}-1)\log p}{(p^{1+\alpha -it}-1)(p^{1+\beta +it}-1)+p^{1+\alpha +\beta}-1}=\frac{\log p}{p}+O\left( \frac{\log p}{p^{\frac{3}{2}}} \right)
\end{equation}
uniformly for $|\alpha -it|, |\beta +it| \ll 1/\log X$, by Lemma 3.8 we have
\begin{equation}
\label{44}
\begin{aligned}
H_{1}&=Y_{q}(\alpha ,\beta ,t)\sum_{\underset{(p_{1}\cdots p_{m-1}, q)=1}{p_{1}\cdots p_{m-1}\leq X}}\mu ^{2}(p_{1}\cdots p_{m-1})\prod _{r=1}^{m-1}\frac{(p_{r}^{1+\alpha +\beta}-1)\log p_{r}}{(p_{r}^{1+\alpha -it}-1)(p_{r}^{1+\beta +it}-1)+(p_{r}^{1+\alpha +\beta}-1)}  \\
& \quad \times \left\{ \sum _{p_{m}\leq X/p_{1}\cdots p_{m-1}}\frac{\log p_{m}}{p_{m}} \int _{1}^{X/p_{1}\cdots p_{m}}\frac{G_{0}(\alpha -it, up_{1}\cdots p_{m})G_{m}(\beta +it, up_{1}\cdots p_{m})}{u} du  \right. \\
& \left. \quad \quad \quad   +O\left(  \sum _{\underset{p_{m}|qp_{1}\cdots p_{m-1}}{p_{m}\leq X/p_{1}\cdots p_{m-1}}}\frac{(p_{m}^{1+\alpha +\beta}-1)\log p_{m}}{(p_{m}^{1+\alpha -it}-1)(p_{m}^{1+\beta +it}-1)+(p_{m}^{1+\alpha +\beta}-1)} \cdot \frac{1}{\log ^{m+1}X}  \right) \right. \\
& \quad \quad \quad  \left. +O\left( \sum _{p_{m}\leq X/p_{1}\cdots p_{m-1}}\frac{\log p_{m}}{p_{m}^{\frac{3}{2}}} \cdot \frac{1}{\log ^{m+1}X} \right) \right\}  .
\end{aligned}
\end{equation}
By Lemma 3.4 and (\ref{Ges}), (\ref{Gdash}), the contribution of the error terms in (\ref{44}) to $H_{1}$ is at most
\begin{align*}
\sum _{p_{1}, \ldots ,p_{m-1}\leq X}\prod _{r=1}^{m-1}\frac{\log p_{r}}{p_{r}}\left( \sum _{\underset{p_{m}|qp_{1}\cdots p_{m-1}}{p_{m}\leq X}}\frac{\log p_{m}}{p_{m}} +O(1) \right)\frac{1}{\log ^{m+1}X} \ll \frac{\log (\log q \log X)}{\log ^{2}X}.
\end{align*}
Hence
\begin{align*}
H_{1}&=Y_{q}(\alpha ,\beta ,t)\sum_{\underset{(p_{1}\cdots p_{m-1}, q)=1}{p_{1}\cdots p_{m-1}\leq X}}\mu ^{2}(p_{1}\cdots p_{m-1})\prod _{r=1}^{m-1}\frac{(p_{r}^{1+\alpha +\beta}-1)\log p_{r}}{(p_{r}^{1+\alpha -it}-1)(p_{r}^{1+\beta +it}-1)+(p_{r}^{1+\alpha +\beta}-1)}  \\
& \quad \quad  \times \sum _{p_{m}\leq X/p_{1}\cdots p_{m-1}}\frac{\log p_{m}}{p_{m}} \int _{1}^{X/p_{1}\cdots p_{m}}\frac{G_{0}(\alpha -it, up_{1}\cdots p_{m})G_{m}(\beta +it, up_{1}\cdots p_{m})}{u} du \\
& \quad +O\left( \frac{\log (\log q \log X)}{\log ^{2}X} \right).
\end{align*}
By Prime Number Theorem, we have
\begin{align*}
&  \sum _{p_{m}\leq X/p_{1}\cdots p_{m-1}}\frac{\log p_{m}}{p_{m}} \int _{1}^{X/p_{1}\cdots p_{m}}\frac{G_{0}(\alpha -it, up_{1}\cdots p_{m})G_{m}(\beta +it, up_{1}\cdots p_{m})}{u} du \\
& \quad =\int _{1}^{X/ p_{1}\cdots p_{m-1}}\frac{1}{u_{m}}\int _{1}^{X/p_{1}\cdots p_{m-1}u_{m}}\frac{G_{0}(\alpha -it, up_{1}\cdots p_{m-1}u_{m})G_{m}(\beta +it, up_{1}\cdots p_{m-1}u_{m})}{u} du du_{m} \\
& \quad \quad  +O\left( \frac{1}{\log ^{m+1}X} \right)
\end{align*}
(see   (4.36) of \cite{F}).  Hence
\begin{equation}
\label{45}
\begin{aligned}
H_{1}&=Y_{q}(\alpha ,\beta ,t)\sum_{\underset{(p_{1}\cdots p_{m-1}, q)=1}{p_{1}\cdots p_{m-1}\leq X}}\mu ^{2}(p_{1}\cdots p_{m-1})\prod _{r=1}^{m-1}\frac{(p_{r}^{1+\alpha +\beta}-1)\log p_{r}}{(p_{r}^{1+\alpha -it}-1)(p_{r}^{1+\beta +it}-1)+(p_{r}^{1+\alpha +\beta}-1)}  \\
& \quad  \times \int _{1}^{X/ p_{1}\cdots p_{m-1}}\frac{1}{u_{m}}\int _{1}^{X/p_{1}\cdots p_{m-1}u_{m}}\frac{G_{0}(\alpha -it, up_{1}\cdots p_{m-1}u_{m})G_{m}(\beta +it, up_{1}\cdots p_{m-1}u_{m})}{u} du du_{m} \\
& \quad +O\left( \frac{\log (\log q \log X)}{\log ^{2}X} \right).
\end{aligned}
\end{equation}
Applying this argument inductively to the sum over  $p_{m-1}, \ldots ,p_{1}$ in this order and using Lemma 3.6, we have
\begin{equation}
\label{46}
\begin{aligned}
H_{1}&=Y_{q}(\alpha ,\beta ;t)\int _{1}^{X}\frac{1}{u_{1}}\int _{1}^{X/u_{1}}\frac{1}{u_{2}} \cdots \int _{1}^{X/ u_{1}\cdots u_{m-1}}\frac{1}{u_{m}} \\
& \quad \quad  \times  \int _{1}^{X/u_{1}\cdots u_{m}} \frac{G_{0}(\alpha -it, uu_{1}\cdots u_{m}) G_{m}(\beta +it, uu_{1}\cdots u_{m})}{u} du du_{m}\cdots du_{1} \\
& \quad +O\left( \frac{\log (\log q \log X)}{\log ^{2}X} \right) \\
&= Y_{q}(\alpha ,\beta ;t)\int _{1}^{X}\frac{G_{0}(\alpha -it, u)G_{m}(\beta +it, u)\log ^{m}u}{m! \; u} du +O\left( \frac{\log (\log q \log X)}{\log ^{2}X} \right).
\end{aligned}
\end{equation}
It follows from the computation of  p.538 of Feng's paper \cite{F} that the integral in (\ref{46}) is equal to 
\[
\frac{1}{\log X} \int _{0}^{1}\frac{V_{0}(\alpha -it, u)V_{m}(\beta +it, u)u^{m}}{m!} du.
\]
Substituting (\ref{42}), (\ref{46}) into (\ref{41.5}) and replacing $Y_{q}(\alpha ,\beta ;t)$ with $Y_{q}(it, -it; t)=F(q,1)$, we obtain (\ref{48}).      
\end{proof}

\begin{lem}
For $1\leq m_{1}, m_{2} \leq I$, we have
\begin{equation}
\label{55}
\begin{aligned}
& \sum _{\underset{(j,q)=1}{j \leq X}}\frac{\mu ^{2}(j)F(j, 1+\alpha +\beta )}{jF(j, 1+\alpha -it)F(j, 1+\beta +it)} \\
& \quad \times G_{m_{1}}(\alpha -it, j) \sum _{p_{1}\cdots p_{m_{1}}|j}\log p_{1}\cdots \log p_{m_{1}} G_{m_{2}}(\beta +it, j) \sum _{p_{1}\cdots p_{m_{2}}|j}\log p_{1}\cdots \log p_{m_{2}}  \\
&\quad =\left( \sum _{k=0}^{\min \{ m_{1}, m_{2} \}}k! \binom{m_{1}}{k}\binom{m_{2}}{k} \right) \frac{F(q, 1)}{\log X}\int _{0}^{1}\frac{V_{m_{1}}(\alpha -it, u)V_{m_{2}}(\beta +it, u)u^{m_{1}+m_{2}}}{(m_{1}+m_{2})!} du \\
& \quad \quad +O\left( \frac{\log (\log q \log X)}{\log ^{2}X} \right)
\end{aligned}
\end{equation}
uniformly for $|\alpha -it|, |\beta +it| \ll 1/\log X$. Here, $F(q, 1)=\prod _{p|q}(1-1/p)$ and $V_{m}(\alpha ,u)$ $(m=1, \ldots ,I)$ are defined by (\ref{V0}), (\ref{V1})-(\ref{VI}).
\end{lem}
\begin{proof}
Except for the adaption of Lemma 3.7, this lemma can be proved by the same argument as in the proof of Lemma 3.10. Hence we only describe the outline. By Lemma 3.7, the left hand side of (\ref{55}) equals 
\begin{equation} 
\label{49}
\begin{aligned}
& \sum _{k=1}^{\min \{m_{1}, m_{2} \}}k! \binom{m_{1}}{k} \binom{m_{2}}{k}\sum _{\underset{(j,q)=1}{j \leq X}} \frac{\mu ^{2}(j)F(j, 1+\alpha +\beta )}{jF(j, 1+\alpha -it)F(j, 1+\beta +it)}G_{m_{1}}(\alpha -it, j)G_{m_{2}}(\beta +it ,j) \\
& \quad \times \sum _{p_{1} \cdots p_{m_{1}+m_{2}-k|j}} \log ^{2}p_{1} \cdots \log ^{2}p_{k}\log p_{k+1} \cdots \log p_{m_{1}+m_{2}-k}.
\end{aligned}
\end{equation}
Changing the order of summation we have 
\begin{equation}
\label{50}
\begin{aligned}
& \sum _{\underset{(j,q)=1}{j \leq X}} \frac{\mu ^{2}(j)F(j, 1+\alpha +\beta )}{jF(j, 1+\alpha -it)F(j, 1+\beta +it)}G_{m_{1}}(\alpha -it, j)G_{m_{2}}(\beta +it ,j) \\
& \quad \times \sum _{p_{1} \cdots p_{m_{1}+m_{2}-k|j}} \log ^{2}p_{1} \cdots \log ^{2}p_{k}\log p_{k+1} \cdots \log p_{m_{1}+m_{2}-k} \\
& \quad =\sum _{\underset{(p_{1} \cdots p_{m_{1}+m_{2}-k}, q)=1}{p_{1} \cdots p_{m_{1}+m_{2}-k}\leq X}}\log ^{2}p_{1} \cdots \log ^{2}p_{k}\log p_{k+1} \cdots \log p_{m_{1}+m_{2}-k}\;  S(p_{1}, \ldots ,p_{m_{1}+m_{2}-k}),
\end{aligned}
\end{equation}
where 
\begin{align*}
&S(p_{1}, \ldots ,p_{m_{1}+m_{2}-k}) \\
& \quad  =\sum _{\underset{(j, qp_{1}\cdots p_{m_{1}+m_{2}-k})=1}{j \leq X/p_{1}\cdots p_{m_{1}+m_{2}-k}}} \frac{\mu ^{2}(j)}{p_{1}\cdots p_{m_{1}+m_{2}-k}j}    \\
& \quad \quad \quad \quad \quad \times  G_{m_{1}}(\alpha -it, p_{1}\cdots p_{m_{1}+m_{2}-k}j) G_{m_{2}}(\beta +it, p_{1}\cdots p_{m_{1}+m_{2}-k}j)   \\
& \quad \quad \quad \quad \quad   \times \frac{F(p_{1}\cdots p_{m_{1}+m_{2}-k}j, 1+\alpha +\beta ) }{F(p_{1}\cdots p_{m_{1}+m_{2}-k}j, 1+\alpha -it)F(p_{1}\cdots p_{m_{1}+m_{2}-k}j, 1+\beta +it)  }.  \\
\end{align*}
By abelian summation  and Prime Number Theorem combined with Lemmas 3.4, 3.5 and the estimates (\ref{Ges}), (\ref{Gdash}), it follows that 
\begin{equation}
\label{51}
\begin{aligned}
& S(p_{1}, \ldots ,p_{m_{1}+m_{2}-k}) \\
& \quad =Y_{q}(\alpha , \beta ; t)\prod _{r=1}^{m_{1}+m_{2}-k}\left( \frac{1}{p_{r}}+O\left( \frac{1}{p_{r}^{\frac{3}{2}}} \right) \right) \\
& \quad \quad \quad \times \int _{1}^{X/p_{1}\cdots p_{m_{1}+m_{2}-k}}\frac{G_{m_{1}}(\alpha -it, p_{1} \cdots p_{m_{1}+m_{2}-k}u)G_{m_{2}}(\beta +it, p_{1}\cdots p_{m_{1}+m_{2}-k}u)}{u} du \\
& \quad \quad +O\left( \frac{\log \log qX}{\log ^{m_{1}+m_{2}+2}X} \right).
\end{aligned}
\end{equation}
Hence the left hand side of (\ref{50}) equals 
\begin{equation}
\label{52}
\begin{aligned}
& Y_{q}(\alpha ,\beta ;t)\sum _{\underset{(p_{1}\cdots p_{m_{1}+m_{2}-k-1}, q)=1}{p_{1} \cdots p_{m_{1}+m_{2}-k-1}\leq X}}\log ^{2}p_{1}\cdots \log ^{2}p_{k}\log p_{k+1}\cdots \log p_{m_{1}+m_{2}-k-1} \\
& \quad \quad \quad \quad  \times \prod _{r=1}^{m_{1}+m_{2}-k-1}\left( \frac{1}{p_{r}} +O\left( \frac{1}{p_{r}^{\frac{3}{2}}} \right) \right) \; \tilde{S}(p_{1}, \cdots ,p_{m_{1}+m_{2}-k-1}) \\
& \quad +O\left( \frac{\log \log qX}{\log ^{2}X} \right),
\end{aligned}
\end{equation}
where 
\begin{align*}
& \tilde{S}(p_{1}, \cdots ,p_{m_{1}+m_{2}-k-1}) \\
& \quad =\sum _{\underset{(p_{m_{1}+m_{2}-k}, q)=1}{p_{m_{1}+m_{2}-k} \leq X/p_{1}\cdots p_{m_{1}+m_{2}-k-1}}}\log p_{m_{1}+m_{2}-k} \left( \frac{1}{p_{m_{1}+m_{2}-k}}+O\left( p_{m_{1}+m_{2}-k}^{-\frac{3}{2}} \right) \right) \\
& \quad \quad  \times \int _{1}^{X/ p_{1}\cdots p_{m_{1}+m_{2}-k}}\frac{G_{m_{1}}(\alpha -it, p_{1}\cdots p_{m_{1}+m_{2}-k}u)G_{m_{2}}(\beta +it, p_{1}\cdots p_{m_{1}+m_{2}-k}u)}{u} du.
\end{align*}
The contribution of the  error term above to $\tilde{S}$ is $O(1/\log ^{m_{1}+m_{2}+1}X)$.  By Lemma 3.4, the condition $(p_{m_{1}+m_{2}-k}, q)=1$ can be removed with  error term of size $O(\log \log q/ \log ^{m_{1}+m_{2}+1}X)$.  The sum over $p_{m_{1}+m_{2}-k}$ is rewritten as the integral, by using the Prime Number Theorem. Then, a brief computation yields 
\begin{equation}
\label{53}
\begin{aligned}
& \tilde{S}(p_{1}, \cdots ,p_{m_{1}+m_{2}-k-1}) \\
& \quad =\int _{1}^{X/ p_{1}\cdots p_{m_{1}+m_{2}-k-1}}\frac{1}{u_{m_{1}+m_{2}-k}} \\
& \quad \quad \times \int _{1}^{X/p_{1}\cdots p_{m_{1}+m_{2}-k-1}u_{u_{1}+m_{2}-k}} G_{m_{1}}(\alpha -it, p_{1}\cdots p_{m_{1}+m_{2}-k-1}u_{m_{1}+m_{2}-k}u) \\
& \quad \quad \quad  \times G_{m_{2}}(\beta +it, p_{1}\cdots p_{m_{1}+m_{2}-k-1}u_{m_{1}+m_{2}-k}u) \frac{du}{u} du_{m_{1}+m_{2}-k} \\
& \quad \quad +O\left( \frac{\log (\log q \log X)}{\log ^{m_{1}+m_{2}+1}X} \right).
\end{aligned}
\end{equation}
By (\ref{53}),  (\ref{52}) becomes a sum over $p_{m_{1}+m_{2}-k-1}, \ldots ,p_{1}$ multiplied by an integral  in $u$ and $u_{m_{1}+m_{2}-k}$. Repeating this process inductively to the sums over $p_{m_{1}+m_{2}-k-1}, \ldots ,p_{1}$ in this order, we see that (\ref{50}) equals 
\begin{equation}
\label{54}
\begin{aligned}
& Y_{q}(\alpha ,\beta ;  t)\int _{1}^{X} \frac{\log u_{1}}{u_{1}} \cdots \int _{1}^{X/u_{1} \cdots u_{k-1}}\frac{\log u_{k}}{u_{k}} \int _{1}^{X/u_{1}\cdots u_{k}}\frac{1}{u_{k+1}}\cdots \int _{1}^{X/u_{1}\cdots u_{m_{1}+m_{2}-k-1}}\frac{1}{u_{m_{1}+m_{2}-k}} \\
& \quad \times \int _{1}^{X/ u_{1}\cdots u_{m_{1}+m_{2}-k}}G_{m_{1}}(\alpha -it, u_{1}\cdots u_{m_{1}+m_{2}-k}u)G_{m_{2}}(\beta +it, u_{1}\cdots u_{m_{1}+m_{2}-k}u) \frac{du}{u} \\
& \quad \times du_{m_{1}+m_{2}-k}\cdots du_{1} \\
& \quad +O\left( \frac{\log (\log q \log X)}{\log ^{2}X} \right).
\end{aligned}
\end{equation}
By Lemma 3.6, the repeated integral above is rewritten as an integral involving $V_{m}$. By replacing $Y_{q}(\alpha ,\beta ; t)$ with $Y_{q}(it, -it; t)=F(q,1)$, we arrive at the conclusion of the lemma.
\end{proof}
Substituting the asymptotic formulas of Lemmas 3.9-3.11 into (\ref{33}), we obtain the following result. 
\begin{prop}
We have
\begin{equation}
\label{57}
\begin{aligned}
& \sum _{\underset{(hk, q)=1}{h. k \leq X}}\frac{c(h)c(k)(h, k)^{1+\alpha +\beta }}{h^{1+\beta +it}k^{1+\alpha -it}} \\
& \quad =\frac{F(q,1)}{F(q, 1+\alpha -it)F(q, 1+\beta +it)\log X} \left\{ \int _{0}^{1}{\cal F}(\alpha ,\beta ,t, X ;u) du +O\left( \frac{\log (\log q \log X)}{\log X} \right)  \right\} \\
& \quad \quad +E
\end{aligned}
\end{equation}
uniformly for $|\alpha -it|, |\beta +it| \ll 1/\log X$, where $E$ is given by (\ref{34}), 
\begin{equation}
\label{calF}
\begin{aligned}
&{\cal F}(\alpha ,\beta ,t, X ;u) \\
& \quad :=\sum _{m_{1}=0}^{I}\sum _{m_{2}=0}^{I}\sum _{k=0}^{\min \{m_{1}, m_{2} \}}k!\binom{m_{1}}{k} \binom{m_{2}}{k}\frac{V_{m_{1}}(\alpha -it, u)V_{m_{2}}(\beta +it ,u)u^{m_{1}+m_{2}}}{(m_{1}+m_{2})!},
\end{aligned}
\end{equation}
and $V_{m}(\alpha ,u)$ are given by (\ref{V0}), (\ref{V1})-(\ref{VI}). 
\end{prop}
The contribution of $E$ to the whole is evaluated by Lemma 3.3. Hence we omit to add this repeatedly below. Assume $\alpha +\beta \ll 1/\log X$. By Lemma 3.4, 
\begin{align*}
\zeta _{q}(1+\alpha +\beta )&=\zeta (1+\alpha +\beta )\prod _{p|q}\left( 1-\frac{1}{p} \right) \left( 1+O\left( \frac{\log p}{p \log X} \right) \right) \\
&=\prod _{p|q}\left( 1-\frac{1}{p} \right) \zeta (1+\alpha +\beta ) \left(1+ O\left( \frac{\log \log q}{\log X} \right) \right).
\end{align*}
Similarly, 
\[
\frac{1}{F(q, \gamma )}=\left( \prod _{p|q}\left( 1-\frac{1}{p} \right)^{-1} \right)\left(1+ O\left( \frac{\log \log q}{\log X} \right) \right)
\]
for $\gamma =1+\alpha -it, 1+\beta +it$ with $|\alpha -it|, |\beta +it|\ll 1/\log X$.  It is also convenient to remember that the integral of ${\cal F}$ in (\ref{57}) is $O(1)$. Combining these estimates with (\ref{57}), we have
\begin{align*}
& \zeta _{q}(1+\alpha +\beta )\sum _{\underset{(hk, q)=1}{h. k \leq X}}\frac{c(h)c(k)(h, k)^{1+\alpha +\beta }}{h^{1+\beta +it}k^{1+\alpha -it}} \\
& \quad =\frac{\zeta (1+\alpha +\beta )}{\log X} \left\{ \int _{0}^{1}{\cal F}(\alpha ,\beta ,t, X ;u) du +O\left( \frac{\log (\log q \log X)}{\log X} \right)  \right\}.
\end{align*}
Substituting this into (\ref{32.75}) and using 
\[
\zeta (1+\alpha +\beta )=\frac{1}{\alpha +\beta }+O(1), \quad \zeta (1-\alpha -\beta )=-\frac{1}{\alpha +\beta }+O(1) \quad (\alpha +\beta \ll 1/ \log X),
\]
we have 
\begin{equation}
\label{58}
\begin{aligned}
& \sum _{h, k \leq X}\frac{c(h)c(k)}{\sqrt{hk}}\tilde{\Delta}_{\alpha , \beta }(h, k; Q) \left( \frac{k}{h} \right)^{it} \\
& \quad =\frac{1}{\log X}\sum _{q}W \left( \frac{q}{Q} \right)\varphi ^{*}(q)  \bigl\{  I(\alpha ,\beta ,\sigma ,t, q) +I(-\beta ,-\alpha , \sigma ,t, q) \bigr\} \\
& \quad \quad  +O\left( \frac{Q^{2}\log (\log Q \log X)}{\log X} \right),
\end{aligned}
\end{equation}
where 
\begin{equation}
\label{Ialphabeta}
I(\alpha ,\beta , \sigma , t, q):=\left( \frac{q}{\pi} \right)^{\frac{\alpha +\beta}{2}-\sigma +\frac{1}{2}}\frac{\Gamma \left( \frac{1}{4}+\frac{\alpha}{2} \right)}{\Gamma \left( \frac{\sigma +it}{2} \right)} \frac{\Gamma \left( \frac{1}{4}+\frac{\beta}{2} \right)}{\Gamma \left( \frac{\sigma -it}{2} \right)} \frac{1}{\alpha +\beta} \int _{0}^{1}{\cal F}(\alpha , \beta ,t, X ;u) du.
\end{equation}


\section{Completion of  proofs of theorems}
The remaining work is to compute the derivatives by parameters $\alpha$, $\beta$. These parameters are supposed to close to $\sigma -1/2 \pm it$ respectively, where $\sigma =1/2-R/\log QT$. Hence we continue to  assume $\alpha -it, \beta +it \ll 1/ \log X$, where $X=Q^{\theta}$ ($0<\theta <1$).  The following lemma is used to evaluate the contribution of $I(-\beta ,-\alpha ,\sigma ,t, q)$ in (\ref{58}). 
\begin{lem}
Let $R$ be a positive constant and put $\sigma =1/2-R/\log QT$, $\alpha =\sigma -1/2 +it$ for $t \in [T, 2T]$. Then, for non-negative integer $l$ we have
\begin{equation}
\label{62}
\frac{\Gamma  ^{(l)} \left( \frac{1}{4}-\frac{\alpha}{2} \right)}{\Gamma \left( \frac{\sigma -it}{2} \right)}=e^{\frac{R\log T}{\log QT}}(1+o(1))\log ^{l}T.
\end{equation}
\end{lem}
\begin{proof}
We have
\[
\frac{\Gamma ^{(l)} \left( \frac{1}{4}-\frac{\alpha}{2} \right)}{\Gamma \left( \frac{\sigma -it}{2} \right)}=\frac{\Gamma ^{(l)}\left( \frac{1-\sigma -it}{2} \right)}{\Gamma \left( \frac{\sigma -it}{2} \right)}=\frac{\Gamma ^{(l)}\left( \frac{1-\sigma -it}{2} \right)}{\Gamma \left( \frac{1-\sigma -it}{2} \right)}  \frac{\Gamma \left( \frac{1-\sigma -it}{2} \right)}{\Gamma \left( \frac{\sigma -it}{2} \right)}.
\]
By Stirling's estimate for logarithmic derivatives, we have
\[
\frac{\Gamma ^{(l)}\left( \frac{1-\sigma -it}{2} \right)}{\Gamma \left( \frac{1-\sigma -it}{2} \right)} \sim \log ^{l}T.
\]
On the other hand, by Stirling's formula, 
\[
\frac{\Gamma \left( \frac{1-\sigma -it}{2} \right)}{\Gamma \left( \frac{\sigma -it}{2} \right)}=(1+o(1))\exp (\varphi (1-\sigma ,t)-\varphi (\sigma ,t)),
\]
where 
\begin{align*}
\varphi (\sigma ,t)=&\frac{\sigma -it}{2} \left\{ \frac{1}{2}\log \left( \left( \frac{\sigma}{2} \right)^{2}+\left( \frac{t}{2} \right)^{2} \right)+i\arg \left( \frac{\sigma -it}{2} \right) \right\} .
\end{align*}
Using the estimates
\[
\arg \left( \frac{1-\sigma -it}{2} \right), \quad \arg \left( \frac{\sigma -it}{2} \right) =-\frac{\pi}{2} +O\left( \frac{1}{T} \right), 
\]
\[
\log \left( \left( \frac{1-\sigma}{2} \right)^{2}+\left( \frac{t}{2} \right)^{2} \right), \quad  \log \left( \left( \frac{\sigma}{2} \right)^{2}+\left( \frac{t}{2} \right)^{2} \right) =2 \log \frac{t}{2} +O \left( \frac{1}{T^{2}} \right), 
\]
it follows that 
\[
\varphi (1-\sigma ,t)-\varphi (\sigma ,t)=\frac{R\log T}{\log QT}\left( 1+O\left( \frac{1}{\log T} \right) \right).
\]
Combining these estimates we obtain (\ref{62}). 
\end{proof}
Put 
\[
P(x)=\sum _{l} g_{l} x^{l}, \quad Q(x)=P\left( \frac{1}{2}-x \right).
\]
For convenience, put 
\[
\nu (\sigma  ,t)=\sigma -1/2 +it.
\]
Substituting (\ref{58}) into (\ref{30}), we have 
\begin{equation}
\label{71}
\begin{aligned}
& \sum _{q}W \left( \frac{q}{Q} \right)\frac{1}{T} \sum _{\chi (\mathrm{mod} \; q)}^{\quad \quad \flat}\int _{T}^{2T}\left| F(\sigma +it , \chi ) \right|^{2} dt \\
& \quad  \sim \frac{1}{\log X}\sum _{q}W \left( \frac{q}{Q} \right)\varphi ^{*}(q) \\
& \quad \quad \times \frac{1}{T} \int _{T}^{2T}P \left( \frac{1}{\log q}\frac{\partial}{\partial \alpha} \right) P \left( \frac{1}{\log q}\frac{\partial}{\partial \alpha} \right) \left[ I(\alpha ,\beta ,\sigma ,t, q)+ I(-\beta ,-\alpha ,\sigma ,t, q)   \right] _{\tiny \begin{split} &\alpha =\nu (\sigma  ,t) \\ &\beta =\nu (\sigma  ,-t) \end{split} } dt. 
\end{aligned}
\end{equation}
Since
\[
\frac{\partial ^{l}}{\partial \alpha ^{l}}\left( \frac{\Gamma \left( \frac{1}{4}+\frac{\alpha}{2} \right)}{\Gamma \left( \frac{\sigma +it}{2} \right)} \right)_{\alpha =\nu (\sigma ,t)}, \quad \quad  \frac{\partial ^{l}}{\partial \beta ^{l}}\left( \frac{\Gamma \left( \frac{1}{4}-\frac{\beta}{2} \right)}{\Gamma \left( \frac{\sigma -it}{2} \right)} \right)_{\beta =\nu (\sigma ,-t)} \sim \frac{1}{2^{l}}\log ^{l}T,
\]
we have
\begin{equation}
\begin{aligned}
& P \left( \frac{1}{\log q}\frac{\partial}{\partial \alpha} \right) P \left( \frac{1}{\log q}\frac{\partial}{\partial \alpha} \right) I(\alpha ,\beta ,\sigma ,t, q)\;  \vline  _{\; \alpha =\nu (\sigma  ,t), \beta =\nu (\sigma ,  -t) } \\
& \quad  \sim P\left( \frac{\log \frac{qT}{\pi}}{2 \log q} +\frac{1}{\log q} \frac{\partial}{\partial \alpha} \right) P\left( \frac{\log \frac{qT}{\pi}}{2 \log q} +\frac{1}{\log q} \frac{\partial}{\partial \beta} \right)  \left( \frac{1}{\alpha +\beta }\int _{0}^{1} {\cal F}(\alpha ,\beta , t, X; u) du \right)_{\tiny \begin{split} &\alpha =\nu (\sigma  ,t) \\ &\beta =\nu (\sigma  ,-t) \end{split} }.
\end{aligned}
\end{equation}
Hence 
\begin{equation}
\label{lemA}
\begin{aligned}
&  P \left( \frac{1}{\log q}\frac{\partial}{\partial \alpha} \right) P \left( \frac{1}{\log q}\frac{\partial}{\partial \alpha} \right) I(\alpha ,\beta ,\sigma ,t, q)\;  \vline  _{\; \alpha =\nu (\sigma  ,t), \beta =\nu (\sigma  ,-t) } \\
& \quad \sim P\left( \frac{1+\eta }{2 } +\frac{1}{\log q} \frac{\partial}{\partial \alpha} \right) P\left( \frac{1+\eta }{2 } +\frac{1}{\log q} \frac{\partial}{\partial \beta} \right)   \left( \frac{1}{\alpha +\beta }\int _{0}^{1} {\cal F}(\alpha ,\beta , t, X; u) du \right)_{\tiny \begin{split} &\alpha =\nu (\sigma  ,t) \\ &\beta =\nu (\sigma  ,-t) \end{split} }
\end{aligned}
\end{equation}
for $T\asymp Q^{\eta}$, $q \asymp Q$.  In case of  $(\log Q)^{2} \leq T \leq (\log Q )^{A}$, it follows that
\begin{equation}
\label{lemA'}
\begin{aligned}
&  P \left( \frac{1}{\log q}\frac{\partial}{\partial \alpha} \right) P \left( \frac{1}{\log q}\frac{\partial}{\partial \alpha} \right) I(\alpha ,\beta ,\sigma ,t, q)\;  \vline  _{\; \alpha =\nu (\sigma  ,t), \beta =\nu (\sigma  ,-t) } \\
& \quad \sim Q\left(-\frac{1}{\log q} \frac{\partial}{\partial \alpha} \right) Q\left( -\frac{1}{\log q} \frac{\partial}{\partial \beta} \right)   \left( \frac{1}{\alpha +\beta }\int _{0}^{1} {\cal F}(\alpha ,\beta , t, X; u) du \right)_{\tiny \begin{split} &\alpha =\nu (\sigma  ,t) \\ &\beta =\nu (\sigma  ,-t) \end{split} }.
\end{aligned}
\end{equation}
On the other hand, by Lemma 4.1 
\begin{equation}
\begin{aligned}
& P \left( \frac{1}{\log q}\frac{\partial}{\partial \alpha} \right) P \left( \frac{1}{\log q}\frac{\partial}{\partial \alpha} \right) I(-\beta ,-\alpha ,\sigma ,t, q)\;  \vline  _{\; \alpha =\nu (\sigma  ,t), \beta =\nu (\sigma  ,-t) } \\
& \quad \sim e^{\frac{2R \log T}{\log QT}} P\left(1- \frac{\log \frac{qT}{\pi}}{2 \log q} +\frac{1}{\log q} \frac{\partial}{\partial \alpha} \right) P\left(1- \frac{\log \frac{qT}{\pi}}{2 \log q} +\frac{1}{\log q} \frac{\partial}{\partial \beta} \right) \\
& \quad \quad \quad \quad \quad  \quad \quad \quad \quad  \times  \left( \frac{q^{-\alpha -\beta}}{-\alpha -\beta }\int _{0}^{1} {\cal F}(-\beta ,-\alpha , t, X; u) du \right)_{\tiny \begin{split} &\alpha =\nu (\sigma  ,t) \\ &\beta =\nu (\sigma  ,-t) \end{split} }.
\end{aligned}
\end{equation}
Hence 
\begin{equation}
\label{lemB}
\begin{aligned}
&  P \left( \frac{1}{\log q}\frac{\partial}{\partial \alpha} \right) P \left( \frac{1}{\log q}\frac{\partial}{\partial \alpha} \right) I(-\beta ,-\alpha ,\sigma ,t, q)\;  \vline  _{\; \alpha =\nu (\sigma  ,t), \; \beta =\nu (\sigma  ,-t) } \\
& \sim -e^{\frac{2\eta R}{1+\eta}} \\
& \quad \times P\left( \frac{1-\eta }{2 } +\frac{1}{\log q} \frac{\partial}{\partial \alpha} \right) P\left( \frac{1-\eta }{2 } +\frac{1}{\log q} \frac{\partial}{\partial \beta} \right)   \left( \frac{q^{-\alpha -\beta}}{\alpha +\beta }\int _{0}^{1} {\cal F}(-\beta ,-\alpha  , t, X; u) du \right)_{\tiny \begin{split} &\alpha =\nu (\sigma  ,t) \\ &\beta =\nu (\sigma  ,-t) \end{split} }.
\end{aligned}
\end{equation}
for $T\asymp Q^{\eta}$, $q \asymp Q$.  In case of  $(\log Q)^{2} \leq T \leq (\log Q )^{A}$, it follows that
\begin{equation}
\label{lemB'}
\begin{aligned}
&  P \left( \frac{1}{\log q}\frac{\partial}{\partial \alpha} \right) P \left( \frac{1}{\log q}\frac{\partial}{\partial \alpha} \right) I(-\beta ,-\alpha ,\sigma ,t, q)\;  \vline  _{\; \alpha =\nu (\sigma  ,t), \beta =\nu (\sigma  ,-t) } \\
& \sim -e^{\frac{2\eta R}{1+\eta}}  Q\left(-\frac{1}{\log q} \frac{\partial}{\partial \alpha} \right) Q\left( -\frac{1}{\log q} \frac{\partial}{\partial \beta} \right)   \left( \frac{q^{-\alpha -\beta}}{\alpha +\beta }\int _{0}^{1} {\cal F}(-\beta ,-\alpha , t, X; u) du \right)_{\tiny \begin{split} &\alpha =\nu (\sigma  ,t) \\ &\beta =\nu (\sigma  ,-t) \end{split} }.
\end{aligned}
\end{equation}
Hereafter we take the limit $\theta \to 1-0$ and let $X=Q$. It is convenient to keep in mind that 
\[
\frac{\partial ^{m}}{\partial \alpha ^{m}}\frac{\partial ^{n}}{\partial \beta ^{n}} \left( \frac{1}{\alpha +\beta }\int _{0}^{1} {\cal F}(\alpha ,\beta , t, X; u) du \right)_{\tiny \begin{split} &\alpha =\nu (\sigma  ,t) \\ &\beta = \nu (\sigma  ,-t) \end{split} } 
\]
and
\[ \frac{\partial ^{m}}{\partial \alpha ^{m}}\frac{\partial ^{n}}{\partial \beta ^{n}} \left( \frac{q^{-\alpha -\beta}}{\alpha +\beta }\int _{0}^{1} {\cal F}(-\beta ,-\alpha , t, X; u) du \right)_{\tiny \begin{split} &\alpha =\nu (\sigma ,t) \\ &\beta =\nu (\sigma ,-t) \end{split} } 
\]
are independent of $t$ for any $m$, $n$ (see (\ref{calF})) .  Hence it suffices to substitute $\alpha =\beta =\sigma -1/2$. We change parameters by $\alpha =a/ \log X$, $\beta =b/\log X$. Then when $T\asymp Q^{\eta}$,  the condition $\alpha =\beta =\sigma -1/2$ is equivalent to $a=b=-R/(1+\eta)$. Then the right hand side of (\ref{lemA}) equals 
\begin{align*}
& (\log X) P\left( \frac{1+\eta}{2}+\frac{\partial }{\partial a} \right)P\left( \frac{1+\eta}{2}+\frac{\partial }{\partial b} \right)  \left( \frac{1}{a+b}\int _{0}^{1}{\cal F}(a, b; u) du \right) \; \vline _{a=b=-\frac{R}{1+\eta}},
\end{align*}
and the right hand side of (\ref{lemB}) equals 
\begin{align*}
& -e^{\frac{2\eta R}{1+\eta}}(\log X) P\left( \frac{1-\eta}{2}+\frac{\partial }{\partial a} \right)P\left( \frac{1-\eta}{2}+\frac{\partial }{\partial b} \right)  \left( \frac{e^{-a-b}}{a+b}\int _{0}^{1}{\cal F}(-b,-a; u) du \right) \; \vline _{a=b=-\frac{R}{1+\eta}}, 
\end{align*}
where ${\cal F}(a,b;u)={\cal F}(a /\log X, b /\log X , 0, X; u)$.  Substituting these into (\ref{71}), we have the following conclusion.
\begin{prop}
For $T\asymp Q^{\eta}$, we have 
\begin{equation}
\label{74}
\begin{aligned}
& \sum _{q}W\left( \frac{q}{Q} \right)\sum _{\chi (\mathrm{mod}\; q)}^{\quad \quad \flat}\frac{1}{T} \int _{T}^{2T}\left| F(\sigma +it ,\chi) \right|^{2} dt \\
& \quad \sim \psi (Q) \\
& \quad \times \left[  P\left( \frac{1+\eta}{2}+\frac{\partial }{\partial a} \right)P\left( \frac{1+\eta}{2}+\frac{\partial }{\partial b} \right)  \left( \frac{1}{a+b}\int _{0}^{1}{\cal F}(a, b; u) du \right) \; \vline _{a=b=-\frac{R}{1+\eta}} \right. \\
& \left.   \quad \quad -e^{\frac{2\eta R}{1+\eta}} P\left( \frac{1-\eta}{2}+\frac{\partial }{\partial a} \right)P\left( \frac{1-\eta}{2}+\frac{\partial }{\partial b} \right)  \left( \frac{e^{-a-b}}{a+b}\int _{0}^{1}{\cal F}(-b,-a; u) du \right) \; \vline _{a=b=-\frac{R}{1+\eta}} \right]
\end{aligned}
\end{equation}
and in case of $(\log Q)^{2}\leq T \leq (\log Q)^{A}$, we have 
\begin{equation}
\label{75}
\begin{aligned}
& \sum _{q}W\left( \frac{q}{Q} \right)\sum _{\chi (\mathrm{mod}\; q)}^{\quad \quad \flat}\frac{1}{T} \int _{T}^{2T}\left| F(\sigma +it ,\chi) \right|^{2} dt \\
& \quad \sim \psi (Q) Q\left( -\frac{\partial}{\partial a} \right) Q\left( -\frac{\partial}{\partial b} \right) \left( \frac{\int _{0}^{1}{\cal F}(a,b;u)du -e^{-a-b} \int _{0}^{1}{\cal F}(-b,-a;u) du}{a+b} \right)\; \vline _{a=b=-R},
\end{aligned}
\end{equation}
where 
\[
\psi (Q):=\sum_{q}W\left( \frac{q}{Q} \right)\varphi ^{\flat}(q)
\]
and $P(x)=\sum g_{l}x^{l}$ is a real polynomial with $g_{0}=1$, $g_{2n}=0$ $(n\in \mathbb{Z}_{>0})$, $Q(x)=P(1/2-x)$. The function ${\cal F}(a,b;u)$ is defined by
\begin{equation}
\label{Fabu}
\begin{aligned}
& {\cal F}(a,b;u):=\sum _{m_{1}=0}^{I} \sum _{m_{2}=0}^{I}\sum _{k=0}^{\min \{m_{1}, m_{2} \}}k! \binom{m_{1}}{k}\binom{m_{2}}{k}\frac{U_{m_{1}}(a,u)U_{m_{2}}(b,u)u^{m_{1}+m_{2}}}{(m_{1}+m_{2})!},
\end{aligned}
\end{equation}
where
\begin{equation}
\label{U0}
\begin{aligned}
U_{0}(a,u)=&a P_{1}(1-u)+P_{1}^{'}(1-u)  +\sum _{l=2}^{I} \frac{(-1)^{l}}{(l-2)!}\int _{0}^{1-u}P_{l}(1-u-\mu )\mu ^{l-2}e^{-au} d\mu ,
\end{aligned}
\end{equation}

\begin{equation}
\label{U1}
\begin{aligned}
U_{1}(a,u)=&-2P_{2}(1-u)+\sum _{l=3}^{I}\binom{l}{1} \frac{(-1)^{l-1}}{(l-3)!}\int _{0}^{1-u}P_{l}(1-u-\mu )\mu ^{l-3}e^{-au} d\mu ,
\end{aligned}
\end{equation}

\begin{equation}
\label{Um}
\begin{aligned}
U_{m}(a,u)=& aP_{m}(1-u)+P_{m}^{'}(1-u)-\binom{m+1}{m}P_{m+1}(1-u)\\
& \quad \quad +\sum _{l=m+2}^{I}\binom{l}{m} \frac{(-1)^{l-m}}{(l-m-2)!}\int _{0}^{1-u}P_{l}(1-u-\mu )\mu ^{l-m-2}e^{-au} d\mu \quad (2\leq m \leq I-2),
\end{aligned}
\end{equation}

\begin{equation}
\label{UI-1}
\begin{aligned}
U_{I-1}(a,u)=&aP_{I-1}(1-u)+P_{I-1}^{'}(1-u)-IP_{I}(1-u),
\end{aligned}
\end{equation}

\begin{equation}
\label{UI}
\begin{aligned}
U_{I}(a,u)=&aP_{I}(1-u)+P_{I}^{'}(1-u).
\end{aligned}
\end{equation}

\end{prop}
We evaluate  the proportion of critical zeros by combining Proposition 2.1 and Proposition 4.2. To evaluate the proportion of simple critical zeros, we choose $R=0.7150$, $I=3$, $P_{1}(x)=-0.144781x+2.33768x^{2}-1.1929x^{3}$, $P_{2}(x)=1.80598x+0.0466787x^{2}$, $P_{3}(x)=-0.332995x$, $Q(x)=1-0.955682x$.  Then
\begin{align*}
& 1-\frac{1}{R}\log \left(Q\left( -\frac{\partial}{\partial a} \right) Q\left( -\frac{\partial}{\partial b} \right) \left( \frac{\int _{0}^{1}{\cal F}(a,b;u)du -e^{-a-b} \int _{0}^{1}{\cal F}(-b,-a;u) du}{a+b} \right)\; \vline _{a=b=-R}  \right) \\
& \quad =0.6044\ldots .
\end{align*}
To evaluate the proportion of critical zeros, we choose $R=0.7721$, $I=3$, $P_{1}(x)=x+0.1560x(1-x)-1.4045x(1-x)^{2}-0.0662x(1-x)^{3}$, $P_{2}(x)=2.0409x+0.2661x^{2}$, $P_{3}(x)=-0.0734x$, $Q(x)=1-0.7721x-0.1901(x^{2}/2-x^{3}/3)-3.9627(x^{3}/3-x^{4}/4+x^{5}/5)$ (these are Feng's choice in \cite{F}). Then 
\begin{align*}
&1-\frac{1}{R}\log \left(Q\left( -\frac{\partial}{\partial a} \right) Q\left( -\frac{\partial}{\partial b} \right) \left( \frac{\int _{0}^{1}{\cal F}(a,b;u)du -e^{-a-b} \int _{0}^{1}{\cal F}(-b,-a;u) du}{a+b} \right)\; \vline _{a=b=-R}  \right) \\\
& \quad =0.6107\ldots .
\end{align*}
These results prove Theorem 1.1. Theorem 1.3 is obtained by choosing 
\begin{align*}
R=&0.7150 +0.632539 \eta -0.142758\eta ^{2}+0.0377946\eta ^{3} \\
&-0.0062075 \eta ^{4}+0.000411417\eta ^{5}, 
\end{align*}
$I=3$, $P(x)=1+r(x-(1+\eta)/2)$, $P_{1}(x)=ax+bx^{2}+(1-a-b)x^{3}$, $P_{2}(x)=cx+dx^{2}$, $P_{3}(x)=ex$ with
\begin{align*}
r=&0.955682 -0.690002\eta +0.344604 \eta ^{2}-0.102225 \eta ^{3} \\
& +0.0158217 \eta ^{4}-0.000975675 \eta ^{5}, 
\end{align*}
\begin{align*}
a=&-0.144781+0.889028\eta -0.410202\eta ^{2}+0.0858293 \eta ^{3} \\
&-0.00602004 \eta ^{4}-0.00007205 \eta ^{5}, 
\end{align*}
\begin{align*}
b=&2.33768-1.65646\eta +0.633289\eta ^{2}-0.068012 \eta ^{3} \\
&-0.0118499\eta ^{4}+0.00206328 \eta ^{5},
\end{align*}
\begin{align*}
c=&1.80598-1.62246\eta +0.854232\eta ^{2}-0.267006\eta ^{3} \\
& +0.0478396\eta ^{4}-0.00363412\eta ^{5},
\end{align*}
\begin{align*}
d=& 0.0466787 +0.245593 \eta -0.31053\eta ^{2} +0.281681 \eta ^{3} \\
& -0.0946981\eta ^{4}+0.00981632\eta ^{5}, 
\end{align*}
\begin{align*}
e=& -0.322995+0.280121\eta +0.452439\eta ^{2} \\
&-0.390328\eta ^{3}+0.103507 \eta ^{4}-0.00919719 \eta ^{5}.
\end{align*}
With these values the function $C(\eta)$ in Theorem 1.3 is given by 
\begin{align*}
& 1-\frac{1}{R}\log \left[P\left( \frac{1+\eta}{2}+\frac{\partial }{\partial a} \right)P\left( \frac{1+\eta}{2}+\frac{\partial }{\partial b} \right)  \left( \frac{1}{a+b}\int _{0}^{1}{\cal F}(a, b; u) du \right) \; \vline _{a=b=-\frac{R}{1+\eta}}  \right. \\
& \left. \quad \quad  -e^{\frac{2\eta R}{1+\eta}} P\left( \frac{1-\eta}{2}+\frac{\partial }{\partial a} \right)P\left( \frac{1-\eta}{2}+\frac{\partial }{\partial b} \right)  \left( \frac{e^{-a-b}}{a+b}\int _{0}^{1}{\cal F}(-b,-a; u) du \right) \; \vline _{a=b=-\frac{R}{1+\eta}} \right].
\end{align*}


\section{Appendix. The validity of (\ref{31}) with (\ref{32.5}) in a larger domain}
In a paper of Conrey, Iwaniec and Soundararajan \cite{CIS3}, the proof of the asymptotic formula (\ref{31}) with (\ref{32}) started from an approximate functional equation for $\Lambda(1/2+\alpha ,\chi )\Lambda (1/2+\beta ,\overline{\chi})$. For complex numbers $\alpha$, $\beta$ with $\alpha +\beta \neq 0$, define
\begin{equation}
\label{Vtilde}
\tilde{V}_{\alpha ,\beta}(s)=\Gamma \left( \frac{s+1/2+\alpha }{2} \right) \Gamma \left( \frac{s+1/2+\beta }{2} \right) \left( 1-\left( \frac{2s}{\alpha +\beta } \right)^{2} \right).
\end{equation}
For any positive real number $x$, put 
\begin{equation}
\label{V}
V_{\alpha ,\beta}(x)=\frac{1}{2\pi i}\int _{(1)}\tilde{V}_{\alpha ,\beta}(s)x^{-s}\frac{ds}{s}.
\end{equation}
Then,
\begin{lem}[Lemma 1 of \cite{CIS3}, Approximate Functional Equation]
We have
\begin{equation}
\label{AFE}
\Lambda (1/2+\alpha ,\chi )\Lambda (1/2+\beta ,\overline{\chi})=S(\alpha ,\beta ;\chi)+S(-\beta ,-\alpha ;\chi),
\end{equation}
where
\[
S(\alpha ,\beta ;\chi)=\left( \frac{q}{\pi} \right)^{\frac{\alpha +\beta}{2}}\sum _{m,n}\frac{\chi (m)\overline{\chi}(n)}{m^{1/2 +\alpha }n^{1/2 +\beta }}V_{\alpha ,\beta }\left( \frac{\pi mn}{q} \right).
\]
\end{lem}
In the computation of the second moment of (\ref{AFE}), they trancated the Dirichlet series using estimates
\begin{equation}
\label{CISlarge}
V_{\alpha ,\beta}(x)\ll \exp (-\tau x)
\end{equation}
for  $x>1$ and 
\begin{equation}
\label{CISsmall}
V_{\alpha ,\beta}(x)=\Gamma \left(\frac{1}{4}+\frac{\alpha}{2} \right)  \left(\frac{1}{4}+\frac{\beta}{2} \right) +O(x^{\frac{1}{2}-\varepsilon})
\end{equation}
for $0<x\leq 1$, where $\tau$ is some suitable positive constant. The implicit constants above are dependent on $\alpha$, $\beta$, but if the absolute values of $\alpha$ and $\beta$ are small enough, the implicit constants do not interfere with the discussion.  We, however, need to consider the case that $|\Re (\alpha )|, |\Re (\beta )|<1/\log Q$, $(\log Q)^{2}\leq |\Im (\alpha )|, |\Im (\beta)|<(\log Q)^{A}$,  and in this case the size of $\alpha$, $\beta$ might cause some problem in order estimates. Instead of (\ref{CISlarge}) and (\ref{CISsmall}), we use the following hybrid bounds which take both $x$ and $\Im (\alpha) (\sim -\Im (\beta))$ into account. 
\begin{lem}
Suppose that $\alpha ,\beta \in \mathbb{C}$ satisfy
\begin{equation}
\label{A1}
|\alpha -iT|, |\beta +iT|\leq \delta _{1}, \quad |\alpha +\beta |\geq \delta _{2}
\end{equation}
for $T>0$ and $0< \delta _{1},\delta _{2}\leq 1/100$. Then we have
\begin{equation}
\label{A6}
V_{\alpha ,\beta}(x) \ll \left| \Gamma \left( \frac{1/2+\alpha}{2} \right) \Gamma \left( \frac{1/2+\beta}{2} \right) \right| \delta _{2}^{-2}\exp \left(-\frac{x}{2}\right)
\end{equation}
for $x>T$, and 
\begin{equation}
\label{A5} 
V_{\alpha ,\beta}(x)= \Gamma \left( \frac{1/2+\alpha}{2} \right) \Gamma \left( \frac{1/2+\beta}{2} \right)  \left(1+O_{\varepsilon} (\delta _{2}^{-2}T^{3/2+2\delta _{1}+\varepsilon}x^{1/2-\varepsilon}) \right)
\end{equation}
for $0<x \leq 1$ and any $\varepsilon >0$. The implicit constants in (\ref{A6}), (\ref{A5}) are independent of $\alpha ,\beta$ and $\delta _{1}, \delta _{2}$. 
\end{lem}

\begin{proof}
We first move the path of integration of (\ref{V}) to $\Re (s)=B>1$. Write $s=B+it$. Then by Stirling's approximation, we have
\[
 \Gamma \left( \frac{s+1/2+\alpha}{2} \right)\ll e^{-\frac{\pi}{4}|t+T|}(1+|t+T|)^{B/2-1/4+\delta _{1}/2}, 
 \]
 \[  \Gamma \left( \frac{s+1/2+\beta}{2} \right)\ll e^{-\frac{\pi}{4}|t-T|}(1+|t-T|)^{B/2-1/4+\delta _{1}/2}.
\]
Furthermore, 
\[
1-\left(\frac{2s}{\alpha +\beta} \right)^{2}\ll \delta _{2}^{-2}(B+|t|)^{2}, \quad \frac{1}{s}\ll \frac{1}{B+|t|}.
\]
Combining these estimates we have
\[
V_{\alpha ,\beta}(x)\ll x^{-B}\delta _{2}^{-2}\int _{-\infty}^{\infty}e^{-\frac{\pi}{4}(|t+T|+|t-T|)}(1+|t+T|)^{B/2-1/4+\delta _{1}/2}(1+|t-T|)^{B/2-1/4+\delta _{1}/2}(B+|t|)dt.
\]
Due to symmetry, it suffices to estimate the integral over $[0, \infty)$. Hence 
\[
V_{\alpha ,\beta}(x)\ll \delta _{2}^{-2}x^{-B}(I_{1}+I_{2}),
\]
where
\[
I_{1}:=\int _{0}^{T}e^{-\frac{\pi}{4}(|t+T|+|t-T|)}(1+|t+T|)^{B/2-1/4+\delta _{1}/2}(1+|t-T|)^{B/2-1/4+\delta _{1}/2}(B+|t|)dt,
\]
\[
I_{2}:=\int _{T}^{\infty}e^{-\frac{\pi}{4}(|t+T|+|t-T|)}(1+|t+T|)^{B/2-1/4+\delta _{1}/2}(1+|t-T|)^{B/2-1/4+\delta _{1}/2}(B+|t|)dt.
\]
If $0\leq t \leq T$, then $|t+T|+|t-T|=2T$, $1+|t+T|, 1+|t-T|\ll T$. Hence
\[
I_{1}\ll e^{-\frac{\pi}{2}T}T^{B-1/2+\delta _{1}}\int _{0}^{T}(B+|t|)dt \ll  e^{-\frac{\pi}{2}T}T^{B+1/2+\delta _{1}}(B+T).
\]
If $t\geq T$, then $|t+T|+|t-T|=2t$, $1+|t+T|, 1+|t-T|\leq 1+2t$. Hence 
\[
I_{2}\ll \int _{T}^{\infty}e^{-\frac{\pi}{2}t}(1+2t)^{B-1/2+\delta _{1}}(B+t) dt \ll e^{-\frac{\pi}{2}T}T^{B-1/2+\delta _{1}}(B+T).
\]
Therefore,
\begin{equation}
\label{A2}
V_{\alpha ,\beta}(x)\ll x^{-B}\delta _{2}^{-2}e^{-\frac{\pi}{2}T}T^{B+1/2+\delta _{1}}(B+T).
\end{equation}
On the other hand,
\begin{equation}
\label{A3}
 \Gamma \left( \frac{1/2+\alpha}{2} \right) \Gamma \left( \frac{1/2+\beta}{2} \right) \gg \left| \Gamma \left( \frac{1/2-\delta _{1}+iT}{2} \right) \Gamma \left( \frac{1/2-\delta _{1}-iT}{2} \right) \right| \gg e^{-\frac{\pi}{2}T}T^{-1/2-\delta _{1}}.
\end{equation}
Combining (\ref{A2}) and (\ref{A3}), we have
\begin{equation}
\label{A4}
V_{\alpha ,\beta}(x)\ll \left|   \Gamma \left( \frac{1/2+\alpha}{2} \right) \Gamma \left( \frac{1/2+\beta}{2} \right)  \right| \delta _{2}^{-2}T^{B+1+2\delta _{1}}(B+T)x^{-B}
\end{equation}
for any $B>1$,  $x>1$. Notice that the implicit constant above is independent of $B$. Suppose $x>T$. By choosing $B=x/\log \frac{x}{T}$ in (\ref{A4}). we have 
\[
V_{\alpha ,\beta}(x)\ll \left|   \Gamma \left( \frac{1/2+\alpha}{2} \right) \Gamma \left( \frac{1/2+\beta}{2} \right)  \right| \delta _{2}^{-2} x^{3+\delta _{1}}\exp (-x).
\]
Hence we obtain (\ref{A6}) for $x>T$. Next, suppose $0<x \leq 1$. We move the path of integration  to $\Re (s)=-1/2+\varepsilon$. Then we cross a single pole at $s=0$ with residue  $\Gamma ((1/2+\alpha)/2)\Gamma ((1/2+\beta)/2)$. Therefore,
\[
V_{\alpha ,\beta}(x)= \Gamma \left( \frac{1/2+\alpha}{2} \right) \Gamma \left( \frac{1/2+\beta}{2} \right) +\frac{1}{2\pi i}\int _{(-\frac{1}{2}+\varepsilon )}\tilde{V}_{\alpha ,\beta}(s)x^{-s}\frac{ds}{s}.
\]
The integral above is bounded by the right hand side of (\ref{A4}) with $B=-1/2+\varepsilon$. Thus we obtain (\ref{A5}).
\end{proof}
Let $Q>1$ be sufficiently large and choose $\delta _{1}=\delta _{2}=1/\log Q$. If $T$ satisfies $(\log Q)^{2}\leq T<(\log Q)^{A}$ for some $A>2$, then by (\ref{A6}), we have 
\begin{equation}
\label{A7}
\begin{aligned}
V_{\alpha ,\beta}(x) &\ll \left|   \Gamma \left( \frac{1/2+\alpha}{2} \right) \Gamma \left( \frac{1/2+\beta}{2} \right)  \right| (\log Q)^{2}\exp \left( -\frac{x}{2} \right) \\
&\ll \left|   \Gamma \left( \frac{1/2+\alpha}{2} \right) \Gamma \left( \frac{1/2+\beta}{2} \right)  \right| \exp \left( -\frac{x}{3} \right)
\end{aligned}
\end{equation}
for $x>(\log Q)^{A}$. On the other hand,by (\ref{A5}), we have
\begin{equation}
\label{A8}
V_{\alpha ,\beta}(x)= \Gamma \left( \frac{1/2+\alpha}{2} \right) \Gamma \left( \frac{1/2+\beta}{2} \right) \left(1+O_{\varepsilon}(Q^{\varepsilon}x^{1/2-\varepsilon}) \right)
\end{equation}
for $0<x \leq 1$. By using (\ref{A7}) and (\ref{A8}) instead of (\ref{CISlarge}) and (\ref{CISsmall}) and  making much the same argument as in \cite{CIS3}, we see that the asymptotic formula  (\ref{31}) with (\ref{32.5}) is valid in the domain $|\Re (\alpha)|, |\Re (\beta )|\leq 1/\log Q$, $(\log Q)^{2}\leq |\Im (\alpha)|, |\Im (\beta)|<(\log Q)^{A}$, $\Im (\alpha)\sim -\Im (\beta)$.

\section{Acknowledgements}
This work is partially supported by the JSPS, KAKENHI Grant Number 21K03204 and 24K06697.


\noindent
Kanto Gakuin University, \\
Kanazawa, Yokohama\\
Kanagawa, Japan\\
E-mail address: sono@kanto-gakuin.ac.jp

\end{document}